\documentclass[bj]{imsart}
\input{preamble.tex}

\begin{document}
\begin{frontmatter}

  \title{Determinantal Point Processes in the Flat Limit}
\runtitle{Determinantal Point Processes in the Flat Limit}

\begin{aug}
\author[A]{\fnms{Simon} \snm{Barthelmé}\ead[label=e1]{name.surname@gipsa-lab.fr}},
  \author[A]{\fnms{Nicolas} \snm{Tremblay}\ead[label=e1]{name.surname@gipsa-lab.fr}},
  \author[B]{\fnms{Konstantin} \snm{Usevich}\ead[label=e3]{name.surname@univ-lorraine.fr}}
  \and
  \author[A]{\fnms{Pierre-Olivier} \snm{Amblard}\ead[label=e1]{name.surname@gipsa-lab.fr}}

\address[A]{CNRS, Univ. Grenoble Alpes,  Grenoble INP, GIPSA-lab. \printead{e1}}

  \address[B]{Universit\'{e} de Lorraine and CNRS, CRAN (Centre de Recherche en
    Automatique en Nancy). \printead{e3}}
\end{aug}

  \begin{abstract}
    Determinantal point processes (DPPs) are repulsive point processes where the
    interaction between points depends on the determinant of a positive-semi definite matrix. 
    
    In this paper, we study the limiting process of L-ensembles based on kernel
    matrices, when the kernel function becomes flat (so that every point
    interacts with every other point, in a sense). We show that these limiting
    processes are best described in the formalism of extended L-ensembles and
    partial projection DPPs, and the exact limit depends mostly on the
    smoothness of the kernel function. In some cases, the limiting process is
    even universal, meaning that it does not depend on specifics of the kernel
    function, but only on its degree of smoothness.

    Since flat-limit DPPs are still repulsive processes, this implies that
    practically useful families of DPPs exist that do not require a spatial
    length-scale parameter. 
  \end{abstract}

\end{frontmatter}

\tableofcontents

\section*{Introduction}
\label{sec:intro}
Modeling repulsivity in point patterns is an important problem  in many
applications of random point processes. Classical examples of repulsive point
processes include the localisation of trees in  a forest, the timing of action
potentials in the nervous system, the eigenvalues of random matrices, or
fermionic particles in some potential  \cite{Macchi:CoincidenceApproach}.
Repulsive point processes may also be constructed with some application in mind:
in machine learning, it may be used to improve or accelerate learning~\cite{kulesza2012determinantal}. 

In ML applications of repulsive point processes,
a subset $\X$ of size $m $ needs to be extracted from a ground set $\Omega$ of
size $n$. $\Omega$ may represent for instance a training set, too
large for practical computation, and $\X$ a subset that is in some sense
``representative'' of $\Omega$ for the purposes of training a learning
algorithm. If $\X$ includes too many elements that are similar, it fails to be
representative of the whole of $\Omega$. A solution to this problem  is to
induce repulsivity between the elements sampled, or in other words, to sample
the elements not independently, but with negative correlation \cite{tremblay2019determinantal}.

Determinantal point processes (DPP) are by now perhaps the most famous example of negatively correlated point processes. 
The notion of  diversity in DPP is defined relative to a notion of similarity represented by a
positive-definite kernel. For instance, if the items are vectors in $\R^d$,
similarity may be defined via the squared-exponential (Gaussian) kernel:
\begin{equation}
  \label{eq:squared-exp-kernel}
  \kappa_{\varepsilon}(\vect{x},\vect{y}) = \exp \left( - (\varepsilon \norm{\vect{x}-\vect{y}})^2 \right)
\end{equation}
Here $\vect{x}$ and $\vect{y}$ are two items, and similarity is a decreasing
function of distance.


The class of DPPs can be separated into two subclasses: a large subclass called
\emph{L-ensembles} that contains the DPPs that can sample the empty set (the
probability of sampling the empty set is strictly positive); and a much smaller
class made up of those DPPs that cannot (the probability is strictly zero). Precise definitions are to be found in section~\ref{sec:definitions}. 

By definition, an L-ensemble based on the $n\times n$ kernel matrix $\bL=[\kappa_{\varepsilon}(\vect{x}_i,\vect{x}_j)]_{i,j}$ is a distribution over random subsets $\X$ such that:
\[ \Proba(\X) \propto \det [\kappa_{\varepsilon}(\vect{x}_i,\vect{x}_j)]_{\vect{x}_i,\vect{x}_j \in \X^2} \]
If two or more points in $\X$ are very similar (in the sense of the kernel function), then the matrix $\bL_\X = [\kappa_{\varepsilon}(\vect{x}_i,\vect{x}_j)]_{\vect{x}_i,\vect{x}_j \in \X^2}$
has rows that are nearly collinear and the determinant is small (see fig.~\ref{fig:illus-dpp}). This in turns
makes it unlikely that such a set $\X$ will be selected by the L-ensemble. 
For instance, for the points shown in fig.~\ref{fig:illus-dpp}, and $\varepsilon = 10$, the subset $\{ a,b,c\}$ gives a
kernel matrix 
\[ \bL_{\{ a,b,c \}} =
  \begin{pmatrix}
    1 & 0.61 & 0.8 \\
    0.61 & 1 & 0.71 \\
    0.8 & 0.71 & 1
  \end{pmatrix},
\]
with determinant equal to $0.18$. The subset $\{d,e,f\}$ gives a kernel matrix 
\[ \bL_{\{ d,e,f \}} =
  \begin{pmatrix}
    1 & 0 & 0.14 \\
    0 & 1 & 0.01 \\
    0.14 & 0.01 & 1
  \end{pmatrix},
\]
with determinant equal to $0.98$. Accordingly, the second set is $0.98/0.18
\approx 5.4 $ times more likely to be sampled.

\begin{figure}[!htbp]
  \begin{center}
    \includegraphics[width=10cm]{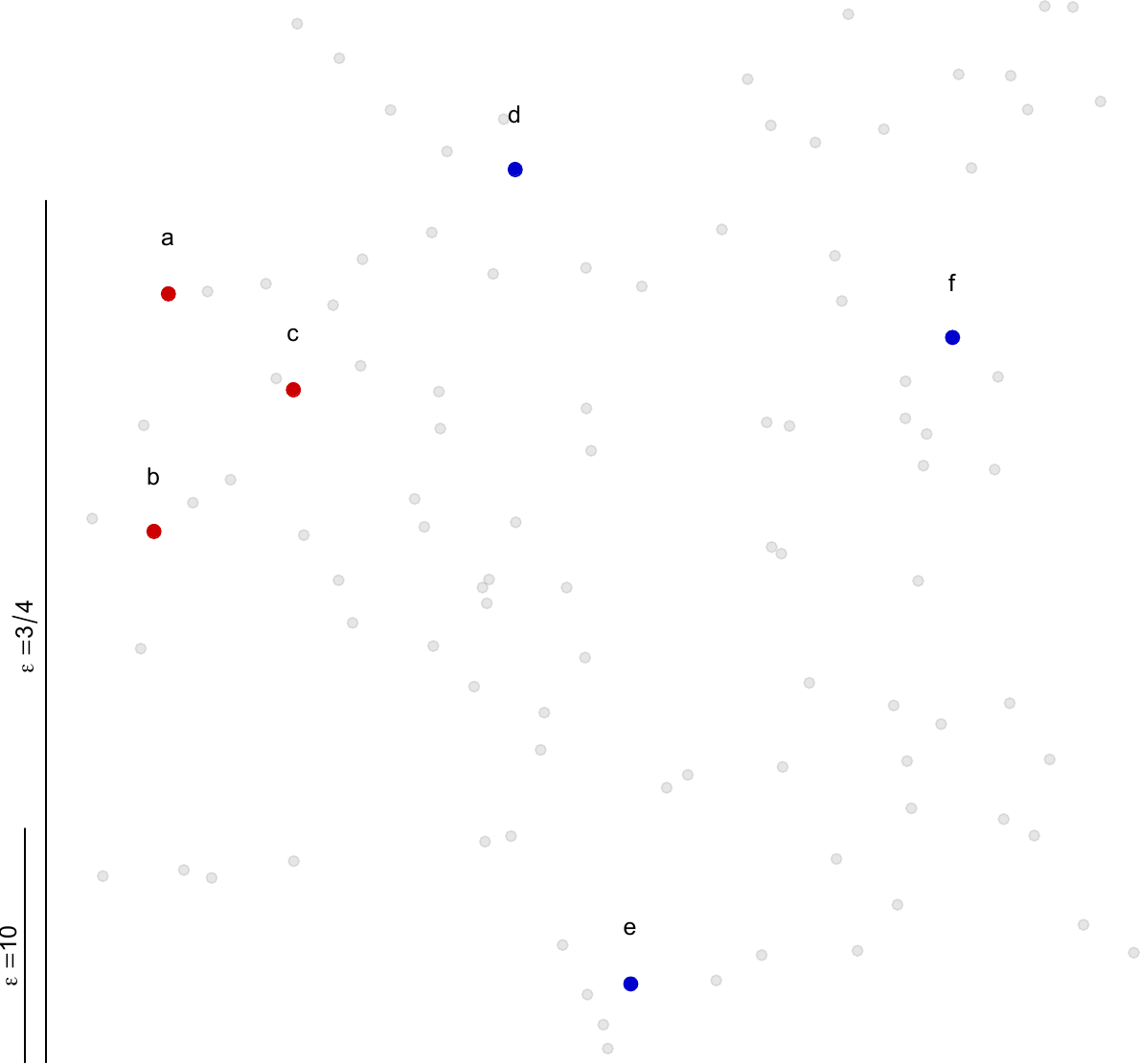}
  \end{center}
\caption{
    L-ensembles generate random subsets with probability
    proportional to the determinant of a kernel matrix. The ground set $\Omega$
    represents the items to sample from: in this figure the points in light
    gray. Two possible subsets of size 3 are represented in blue and red,
    respectively. An L-ensemble may be defined using the Gaussian kernel (eq.
    \ref{eq:squared-exp-kernel}), for instance, and $\varepsilon$ controls the
    length-scale of the kernel (the ``standard deviation'' of the Gaussian
    kernel equals $1/2\sqrt{\varepsilon}$, represented by the two vertical bars
    on the left). The
    set $X = \{a,b,c\}$ contains points that are much closer together than the set
    $X' = \{d,e,f\}$: accordingly, the kernel matrix formed from $X'$ is much better
    conditioned than the one formed from $X$, which is reflected in the determinant (see text). An L-ensemble is therefore much
    more likely to sample $X'$ than $X$. }
  \label{fig:illus-dpp} 
\end{figure}

\begin{figure}
\begin{center}

  \includegraphics[width=12cm]{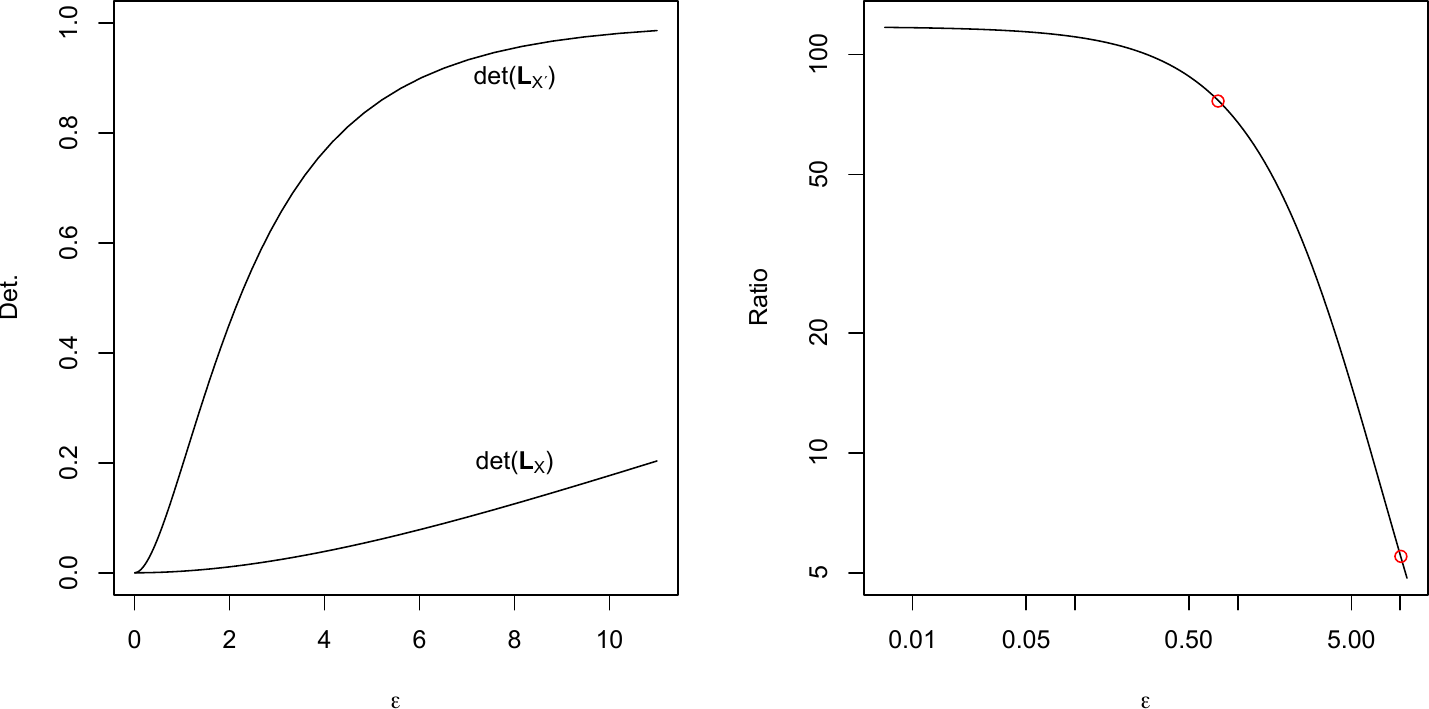}
  \end{center}

  \caption{
In this article, we study the limit of L-ensembles as $\flatlim$, meaning that
the length-scale of the kernel goes to infinity. Although all kernel
matrices are equal to the constant matrix in that limit, and all
determinants go to 0, \emph{ratios} of two determinants go to a fixed quantity.
This is what the figure shows: the left-hand part shows the determinants of
the two kernel matrices from fig. \ref{fig:illus-dpp}  corresponding to $\X$ and $\X'$, as as function of
$\varepsilon$. The right-hand part shows their ratio. The two red dots are for $\varepsilon=10$ and
$\varepsilon=3/4$.
As $\flatlim$, set $\X'$  is roughly 100 times more
    likely than set $\X$ to be sampled. }
  \label{fig:illus-dpp-flat-lim} 
\end{figure}

Importantly, how fast similarity decreases with distance
is determined by the inverse-scale parameter $\varepsilon$. Like other kernel
methods, L-ensembles are plagued with hyperparameters and finding the ``right'' value
for $\varepsilon$ is no easy task. Partial answers to this difficulty may be
obtained via the study of the so-called ``flat limit'', originally studied by
Driscoll \& Fornberg  in Radial Basis
Function interpolation \cite{driscoll2002interpolation}, which simply consists in taking $\varepsilon \rightarrow 0$ in
eq. \eqref{eq:squared-exp-kernel} (or similar kernels).

This paper addresses the question of the behaviour of L-ensembles based on similarity kernels for which $\varepsilon \rightarrow 0$. To this end, we build upon the work in~\cite{BarthelmeUsevich:KernelsFlatLimit}, where general results on the spectral properties of kernel matrices are established in the flat limit.

\subsection*{Contributions}
\label{sec:contrib}

The flat limit is best described in the formalism of extended L-ensembles and partial projection DPPs that we introduced in~\cite{paper1}. In a nutshell, extended L-ensembles provide a unified description of DPPs: whereas not
all DPPs are L-ensembles, all DPPs are extended L-ensembles. In addition, they
let us write easy-to-understand, explicit formulas for joint probabilities even in cases where the DPP at hand is not an L-ensemble. Partial projection DPPs (pp-DPPs) refer to the set of DPPs that are not L-ensembles nor projection DPPs. For instance, the size of such pp-DPPs is necessarily non-null and non-constant. 
Section~\ref{sec:def-DPPs} recalls the necessary definitions and properties of extended L-ensembles and pp-DPPs. 


With these definitions in hand, the next sections describe our results on the
flat limit of DPPs; to be precise, these concern the study of the limiting process of an L-ensemble based on a kernel matrix, as $\varepsilon$ tends to zero.  We show the following results:
\begin{itemize}
\item Surprisingly, in the flat limit, such L-ensembles stay well-defined (see fig. \ref{fig:illus-dpp-flat-lim} for an intuitive explanation of why that occurs)
\item The limiting process depends mostly on the smoothness of the kernel function 
\item In particular cases (depending on the dimension $d$), they exhibit universal
  limits, i.e. all kernels within the same smoothness class lead to the same
  limiting L-ensemble 
\end{itemize}

As an example of our results, we can prove the following (the notation is made precise later):
let $\Omega \subset \R$ (a finite set of points on the real line), and $\X$ an L-ensemble on $\Omega$. Let $\kappa_{\varepsilon}$ be a kernel
function that is $C^\infty$ in both $x$ and $y$ at $\vect{0}$ and analytic in $\varepsilon$  (e.g., the Gaussian). Pick
an odd integer $p<2|\Omega|-1$. Then, applying Thm~\ref{thm:Xe_varying_multivariate}, as $\flatlim$ the L-ensemble based on the matrix 
$[\varepsilon^{-p}
\kappa_{\varepsilon}({x_i},{x_j})]_{{x}_i,{x}_j \in \Omega^2}$
has the law:
\begin{align}
\label{eq:first_example}
     p\left(\X= \{x_1, \ldots, x_m\} \right) =
  \begin{cases}
    \frac{1}{Z} \prod_{i<j} (x_i -
    x_j)^2 & \mbox{\ if\ } m = \frac{p+1}{2}, \\
    0& \mbox{\ otherwise}.
  \end{cases}
\end{align}
On the other hand, if the kernel function is only once differentiable in
$\vect{x}$ and $\vect{y}$ at 0, e.g.
with $\kappa_\varepsilon(x,y) = \exp(-\varepsilon|x-y|)$, then taking the limit
of the L-ensemble based on the matrix $[\varepsilon^{-1}
\kappa_{\varepsilon}({x}_i,{x}_j)]_{({x}_i,{x}_j) \in
  \Omega^2}$ we obtain a different process, with joint probability:
\[     p\left(\X= \{ x_1, \ldots, x_m \} \right) =
  \begin{cases}
    \frac{1}{Z} \gamma^m \prod_{i=1}^{m-1} (x_{i+1} -
    x_i), & \mbox{\ if\ } m \geq 1, \\
    0, & \mbox{\ otherwise},
  \end{cases}
\]
where we have ordered the points so that $x_1 \leq x_2 \leq \ldots \leq x_m$.
Whereas the previous limit was completely universal, in the sense that the
limiting distribution is the same for all $C^\infty$ kernels, this other limit
is almost universal, but not quite: the limit is the same for all $C^1$ kernels,
except for the value of  $\gamma\in\mathbb{R}$ which depends on the kernel. 

Our results are much more general, and the general case involves some
subtleties. The main (and most general) results on the flat
limit are Theorems~\ref{thm:general-case-smooth-fixed-size} and~\ref{thm:Xe_varying_multivariate}, but the statements require that we
set up a bit of notation. 

Because the results require a bit of background to explain properly, we show in
fig. \ref{fig:teaser} a teaser meant to motivate the reader to pursue reading at
least until section \ref{sec:universal-limits}, where the key to the mystery is
revealed. The teaser shows counter-intuitive behaviour of L-ensembles in the flat limit
(in dimension 2). 

\begin{figure}
   \begin{center}

     \includegraphics[width=8cm]{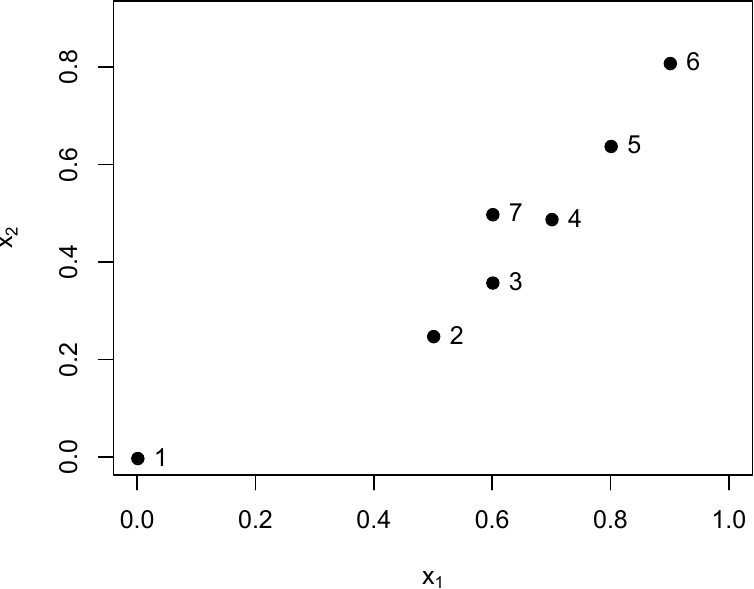}
  \caption{Suppose a (fixed-size) L-ensemble is used to sample 6 of the 7 labelled
    points shown on the figure. With a Gaussian kernel, as $\flatlim$, the set
    $X = \{1,2,3,4,5,6\}$ has a probability 0 of being sampled, while the set
    $X' = \{2,3,4,5,6,7\}$, which is less spread-out, has a small but non-zero
    probability of being sampled. With an exponential kernel, on the other hand,
    both sets have a non-zero probability of being sampled, but in this case $X$
    is much \emph{more} likely to be sampled that $X'$. The explanation for that
    counter-intuitive behaviour is to be found in section
    \ref{sec:universal-limits}.} \label{fig:teaser}
   \end{center}
\end{figure}

The limitations of our results are as follows. We focus on stationary kernels,
of the form $k_\varepsilon(\vect{x},\vect{y}) = f(\varepsilon \norm{x-y})$. The
results can be extended to nonstationary kernels of the form $k(\varepsilon
\vect{x}, \varepsilon \vect{y})$, following the results in
\cite{BarthelmeUsevich:KernelsFlatLimit}, but may be less easy to interpret.
Another limitation is that we only look at finite DPPs, leaving aside the continuous case. All results
should extend to continuous DPPs on a compact subset of $\R^d$, with the
appropriate change in notation. The case of continuous DPPs on a \emph{non-compact} subspace of $\R^d$ appears to us harder to deal with.

\subsection*{Practical implications}

The practical-minded reader might object to the abstract nature of this work.
However, we stress that flat limits are an elegant way of partially answering
the questions of hyper-parameter tuning, and, to a lesser extent, the choice of
similarity function.

One outcome of this work is that as $\flatlim$, DPPs have limits that
are sensible, repulsive and so should behave reasonably in applications. In
particular, the ``default'' distance-based DPPs suggested in \cite{paper1} and
the multivariate orthogonal ensembles used in \cite{bardenet2020monte} are two
such limits. 
One
advantage of directly sampling from the limiting DPP is that there is no spatial
scaling parameter to select. The only one that remains is how many points
one wishes to sample. This assumes of course that one has chosen a particular
kernel function, which leads us to our second point. 

The second conclusion of our work is that what the exact kernel is, matters much
less than what its smoothness order is. If one were to speculate based on the
results in the unidimensional case, kernels with low regularity lead to mostly
local repulsion whereas kernels with high regularity lead to a more global form
of repulsion; and this is borne out as well by some numerical evidence. Kernels
with high regularity lead to some surprising long-distance repulsiveness
properties, as fig. \ref{fig:teaser} illustrates.

\subsection*{Structure of the paper}
We begin with some definitions and background in section \ref{sec:definitions}.
Section \ref{sec:def-DPPs} introduces extended L-ensembles and partial-projection DPPs 
and gives some major properties.

For clarity, flat limit results are given in increasing order of complexity. We
begin with results on the limits of \emph{fixed-size} L-ensembles (the ``k-DPPs'' of
\cite{KuleszaTaskar:FixedSizeDPPs}), because these results are much easier to state and
serve as a building block for the case of \emph{variable-size} L-ensembles. Thus,
section \ref{sec:univariate-results} and section \ref{sec:results-multivariate} study fixed-size L-ensembles in the flat limit. For pedagogical reasons, we begin with univariate results (where the points are a
subset of the real line), before giving the results for the multivariate case,
which require some background on multivariate polynomials.

Section \ref{sec:two_general_theorems} gives the results in complete generality,
meaning that they cover the multivariate case in both fixed-size and
varying-size DPPs. Note that all results given in prior sections are
corrolaries of the two main theorems of section \ref{sec:two_general_theorems}. 

Section \ref{sec:practical-consequences} details some practical consequences of
our results, in terms of eliminating hyperparameters. 

\section{Definitions and background}
\label{sec:definitions}
In this section, we start by briefly recalling some background material on kernels, before giving a few determinantal lemmas that will prove useful in the following. We end this section by making explicit what we mean by the convergence of DPPs from asymptotic series. 


\subsection{Kernels, smoothness orders}
\label{sec:intro-kernels}

We only outline the basic concepts needed to express the results from
\cite{BarthelmeUsevich:KernelsFlatLimit}, which our analysis is based on. For more on kernels, the reader is
invited to consult \cite{stein1999interpolation} or \cite{wendland2004scattered}.
A kernel is a positive definite function $\kappa : \R^d \times \R^d \rightarrow \R$. We call the kernel
stationary if $\kappa(\vect{x},\vect{y}) = f(\norm{\vect{x}-\vect{y}}_2)$ for
some function $f$, i.e. it only depends on the (Euclidean) distance between
$\vect{x}$ and $\vect{y}$.
We assume further that $f$ is analytic\footnote{We choose this assumption for
  simplicity, but it can be relaxed to an assumption of differentiability up to a required order. } at 0, and expand it as:
\begin{equation}
  \label{eq:kernel-expansion}
  f(\norm{\vect{x}-\vect{y}}_2) =  f_0 +  f_1 \norm{\vect{x}-\vect{y}}_2 + f_2 \norm{\vect{x}-\vect{y}}_2^2  + f_3 \norm{\vect{x}-\vect{y}}_2^3 +  \ldots
\end{equation}
where $f_i = \frac{f^{(i)}(0)}{i!}$ is the rescaled derivatives at 0 of $f$.
The smoothness order of the kernel is defined with respect to the \emph{odd}
derivatives of $f$ at 0. This relates to the differentiability of
$f(\norm{\vect{x}-\vect{y}}_2)$ in both $\vect{x}$ and $\vect{y}$. For instance,
the function $\norm{\vect{x}-\vect{y}}$ is not differentiable at $\vect{x} = 0$ (merely continuous), while $\norm{\vect{x}-\vect{y}}^2$ is
infinitely differentiable at this point. In general, for an integer $p$, 
$\norm{\vect{x}-\vect{y}}^{2p+1}$ is $p+1$ times differentiable while
$\norm{\vect{x}-\vect{y}}^{2p}$ is infinitely differentiable in both variables.
Therefore, in eq. \eqref{eq:kernel-expansion}, the first non-zero odd derivative
makes a finitely-differentiable term appear, which motivates the following
definition: 
\begin{definition}
  The smoothness order $r$ of a stationary kernel $\kappa(\vect{x},\vect{y}) =
  f(\norm{\vect{x}-\vect{y}}_2)$ is defined as:
  \begin{equation}
    \label{eq:smoothness-order}
    r = \min \{r | f_{2r-1} \neq 0\}
  \end{equation}
  i.e, the smallest $r$ such that the $r$-th \emph{odd} derivative is non-zero. 
\end{definition}

A kernel like the squared-exponential (eq. \eqref{eq:squared-exp-kernel})
depends on the squared distance and so has $r = \infty$. We call such kernels
\emph{completely smooth}. Kernels with finite values of $r$ are called
finitely smooth (f.s.). An example of a kernel with $r=1$ is the exponential
kernel:
\begin{equation}
  \label{eq:exp-kernel}
  \kappa_{\varepsilon}(\vect{x},\vect{y}) = \exp \left( -\varepsilon \norm{\vect{x}-\vect{y}}_2 \right)
\end{equation}
An example of a kernel with $r=2$ is:
\begin{equation}
  \label{eq:r2-kernel}
  \kappa_{\varepsilon}(\vect{x},\vect{y}) = \left(1+\varepsilon \norm{\vect{x}-\vect{y}}_2\right)\exp \left( -\varepsilon \norm{\vect{x}-\vect{y}}_2 \right)
\end{equation}
The Mat{\`e}rn kernels \cite{stein1999interpolation}, popular in spatial statistics, are a generic family of
kernels which have $r$ as a parameter. Other examples of finitely-smooth kernels
can be found in our numerical results, for instance in fig. \ref{fig:convergence-cond-1d}. 

\subsection{Notation, and two determinant lemmas}
\label{sec:determinant-lemmas}
 Let $\matr{A}$ be a $n \times n$ matrix, and
$Y$, $Z$ be two subsets of indices. Then $\matr{A}_{Y,Z}$ is the submatrix
of $\matr{A}$ formed by retaining the rows in $Y$ and the columns in $Z$.  Furthermore, $\matr{A}_{:,Y}$ (resp. $\matr{A}_{Y,:}$) is the matrix made of the full columns (resp. rows) indexed by $Y$. Finally, we
let $\matr{A}_Y = \matr{A}_{Y,Y}$.
Also, for a matrix $\matr{V}$, by $\mspan(\matr{V})$ we denote its column span, and by $\orth(\matr{V})$ the orthogonal complement of $\mspan(V)$. 

We shall need the two following determinantal lemmas in the paper.

\begin{lemma}
  \label{lem:block-det}
  Let $\matr{M} =
  \begin{pmatrix}
    \matr{A} & \matr{U} \\
    \matr{U}^\top & \matr{W}
  \end{pmatrix}
  $,  with $\matr{A}$ invertible.  Then
  \begin{equation}
    \label{eq:block-det}
    \det(\matr{M}) = \det(\matr{A})\det(\matr{W}-\matr{U}^\top\matr{A}^{-1}\matr{U}).
  \end{equation}
\end{lemma}

The next lemma concerns so-called ``saddle-point matrices'', and is proved in
\cite[Appendix A]{BarthelmeUsevich:KernelsFlatLimit}.  

\begin{lemma}[{\cite[Lemma 3.10]{BarthelmeUsevich:KernelsFlatLimit}}]
  \label{lem:det-saddlepoint}
  Let $\bL \in \R^{n \times n}, \bV \in \R^{n \times p}$, 
with $\bV$ of full column rank and $p \leq n$. 
  Let 
  $\bQ \in \R^{n \times (n-p)}$ be an  orthonormal basis for $\orth({\bV})$ (i.e., $\bQ^{\top} \matr{V} = \matr{0}$, $\rank(\bQ) = n-p$).  
  Then:
  \begin{equation}
    \label{eq:det-saddlepoint}
  \det       \begin{pmatrix}
    \bL & \bV \\
    \bV^\top & \matr{0} 
  \end{pmatrix} = (-1)^p \det(\bV^\top\bV)\det(\bQ^\top\bL\bQ).
\end{equation}
\end{lemma}

\subsection{Convergence of DPPs  from asymptotic series}
\label{sec:total-variation-convergence}
We focus here on convergence in law: we say that a random variable $\X_\varepsilon$ converges to a random variable
$\X_\star$ in $\flatlim$ if for all outcomes $A$
\[
\Proba(\X_{\varepsilon}=A) \to \Proba(\X_{\star}=A). 
\]
In the discrete outcome spaces used here, it is equivalent to convergence in total variation ($\lim_{\flatlim} D_{TV}(\X_\varepsilon,\X_\star)=0$), where for discrete random variables $\X$ and $\Y$ defined on the same
space of outcomes, the total variation distance equals:
\begin{equation}
  \label{eq:total-variation}
  D_{TV}(\X,\Y) = \sum_{A} | \Proba(\X=A) - \Proba(\Y=A) |.
\end{equation}
The results from \cite{BarthelmeUsevich:KernelsFlatLimit} provide us with
asymptotic expansions of the determinants involved in the probability mass
functions. To connect asymptotic expansions with  convergence of random variables we shall use the following simple lemma.

\begin{lemma}
  \label{lem:TV-convergence}
  Let $\X_{\varepsilon}$ be a family of discrete random variables (e.g.,
  a discrete point process) with values in
  the finite set $\Phi$. Let \[ \Proba(\X_\varepsilon=X)=\frac{f_\varepsilon(X)}{\sum_{Y \in
        \Phi}f_\varepsilon(Y)}, \]
  where the following asymptotic expansion holds for $f_{\varepsilon}$ and an
  integer $p$, possibly negative:
  \[ f_\varepsilon(X) = \varepsilon^p(f_0(X) + \O(\varepsilon)).\]
Then  $\X_\varepsilon$ converges to the random variable    $\X_\star$ (with values in $\Phi$), defined as
\[
\Proba(\X_\star=X)=\frac{f_0(X)}{\sum_{Y \in \Phi} f_0(Y)}.
\]

\end{lemma}
\begin{proof}
By direct inspection, we have
\[
 \Proba(\X_\varepsilon=X)=\frac{f_\varepsilon(X)}{\sum_{Y \in  \Phi}f_\varepsilon(Y)} = 
  \frac{f_0(X) + \O(\varepsilon)}{ \sum_{Y\in \Phi} ( f_0(Y) + \O(\varepsilon))} \to \frac{f_0(X)}{\sum_{Y \in \Phi} f_0(Y)},
\]
where convergence holds everywhere since  $\Phi$ is a finite set.
\end{proof}

We will also encounter discrete distributions in which the
(unnormalised) probability mass function $f_{\varepsilon}$ may involve different powers of $\varepsilon$. For
instance, consider the random variable $Y_\varepsilon \in \{1,2,3 \}$ with
unnormalised mass function $f_\varepsilon(Y_\varepsilon= 1) = \alpha_1\varepsilon$,
$f_\varepsilon(Y_\varepsilon= 2) = \alpha_2$, and $f_\varepsilon(Y_\varepsilon = 3) =
\alpha_3\varepsilon^{-1}$. What is the law of $Y_\varepsilon$ as $\flatlim$?
After normalisation, we have:
\begin{align*}
  \Proba(Y_\varepsilon = 1) &= \frac{\alpha_1 \varepsilon}{\alpha_1 \varepsilon + \alpha_2  + \alpha_3 \varepsilon^{-1}} = \frac{\alpha_1 \varepsilon^{2}}{\alpha_3 + \O(\varepsilon)} = \O(\varepsilon^2) \\
  \Proba(Y_\varepsilon = 2) &= \frac{\alpha_2 }{\alpha_1 \varepsilon + \alpha_2  + \alpha_3 \varepsilon^{-1}} = \frac{\alpha_2 \varepsilon}{\alpha_3 + \O(\varepsilon)} = \O(\varepsilon) \\
  \Proba(Y_\varepsilon = 3) &= \frac{\alpha_3 \varepsilon }{\alpha_1 \varepsilon + \alpha_2  + \alpha_3 \varepsilon^{-1}} = \frac{\alpha_3 }{\alpha_3 + \O(\varepsilon)} = 1 + \O(\varepsilon) 
\end{align*}
The diverging order wins, and $Y_\varepsilon$ equals 3 almost surely as
$\flatlim$.

This line of reasoning can be easily generalised to obtain the following lemma,
which simply says that the smallest order in $\varepsilon$ always wins:
\begin{lemma}
  \label{lem:TV-convergence-diverging}
  Let $\X_{\varepsilon}$ be a family of discrete random variables with values in
  the finite set $\Phi$. Let $\Proba(\X_\varepsilon=X)=\frac{f_\varepsilon(X)}{\sum_{Y \in
        \Phi}f_\varepsilon(Y)}$,
  where the following series holds for $f$:
  \[ f_\varepsilon(X) = \varepsilon^{\eta_X}(f_0(X) + \O(\varepsilon)).\]
  for some $\eta_X \in \mathbb{Z}$ which may be negative. Let $\eta_{min} = \min_{X \in \Phi}
  \eta_X$ and $\Phi_{min} = \left\{ X | \eta_X = \eta_{min} \right\}$.
  Then $\X_{\varepsilon} \in \Phi_{min}$ almost surely as $\flatlim$. Moreover, $\X_\varepsilon \to \X_\star$, where $\X_\star$ is the random variable with support in $\Phi_{min}$, with $\Proba(\X_\star = X)=\frac{f_0(X)}{\sum_{Y \in \Phi_{min}} f_0(Y)}$. 
\end{lemma}

 \section{DPPs and extended L-ensembles}
 \label{sec:def-DPPs}
 This section recalls basic facts on DPPs and presents a few properties of a novel representation of DPPs, detailed in~\cite{paper1}, called extended L-ensembles. 

 Let $\Omega=\{ \vect{x}_1, \ldots, \vect{x}_n\} \subset \R^d$ be a collection of
 vectors called the \emph{ground set}. A finite point process $\X$ is a random
 subset $\X \subseteq \Omega$. Abusing notation, we sometimes use $\X$ to
 designate the indices of the items, rather than the items themselves. Which one we mean should be clear from context. 
 
 \begin{definition}[Determinantal Point Process]\label{def:dpp}
 	Let $\bK \in \R^{n \times n}$ be
 	a positive semi-definite matrix verifying $\bm{0}\preceq \bK \preceq \bm{I}$. $\X$ is a DPP with
 	marginal kernel $\bK$ if
 	\begin{equation}
 	\label{eq:def-dpp}
 	\forall A\subseteq\Omega\qquad\Proba(A\subseteq\mathcal{X}) = \det \bK_A,
 	\end{equation}
 	where by convention, $\det \bK_{\varnothing} = 1$. 
 \end{definition}

A related point process is given by the class of L-ensembles:
\begin{definition}[L-ensemble]\label{def:dpp_via_L}
	Let $\bL \in \R^{n \times n}$ designate
	a symmetric positive semi-definite matrix. An L-ensemble based on $\bL$ is a point process $\X$ defined as
	\begin{equation}
	\label{eq:def-dpp_via_L}
	\Proba(\X=X) = \frac{\det \bL_X}{\det(\bI+\bL)},
	\end{equation}
	where by convention, $\det \bL_{\varnothing} = 1$.
\end{definition}
A well-known fact~\cite{kulesza2012determinantal} states that any L-ensemble is a DPP (more precisely: an L-ensemble based on a matrix $\bL$ is a DPP with marginal kernel $\bK=\bL(\bI+\bL)^{-1}$ but the contrary is false, such that the class of L-ensembles is a strict subclass of DPPs. For those DPPs that are also L-ensembles, one can conveniently choose, depending on the application, between a marginal-based approach (via the marginal kernel $\bK$ and Eq.~\eqref{eq:def-dpp}) and a density-based approach (via $\bL$ and Eq.~\eqref{eq:def-dpp_via_L}). Unfortunately, for those DPPs that are not L-ensembles, one only has access to the marginal-based approach and even if formulas for the density exist\footnote{A formula due to~\cite{Macchi:CoincidenceApproach} exists in this case but it is unwieldy. See also the discussion around Corollary 1.D.3 in~\cite{Gau20}.}, they are not easy to use. 

In order to remedy this discrepancy, we developed in~\cite{paper1} a novel formalism, called extended L-ensembles, that we here briefly recall and discuss. 
 The extended L-ensemble representation is based on the notion of non-negative pairs (NNP):
 \begin{definition}
 	\label{def:nnp}
 	A Nonnegative Pair, denoted by $\ELE{\bL}{\bV}$, is a pair $\bL \in \R^{n \times n}$, $\bV \in \R^{n
 		\times p}$ of full column rank with $0\leq p\leq n$, such that $\bL$ is symmetric and conditionally positive semi-definite (CPD) with respect to
 	$\bV$ (that is: $\vect{x}^\top\bL\vect{x} \geq 0$ for all $\vect{x}$ verifying $\bV^\top \vect{x}=0$). Wherever a NNP $\ELE{\bL}{\bV}$ appears below, we consistently use the following notation: 
 	\begin{itemize}
 		\item $\bQ \in \R^{n \times p}$ is an
 		orthonormal basis of  $\mspan \bV$, such that $\bI - \bQ\bQ^\top$
 		is a projector on $\orth \bV$
 		\item $\widetilde{\bL} = (\bI - \bQ\bQ^\top)\bL(\bI -
 		\bQ\bQ^\top)\in \R^{n \times n}$. $\tbL$ is symmetric and real, thus diagonalisable in $\mathbb{R}$. Moreover, as $\bL$ is CPD with respect to $\bV$, all eigenvalues  of $\widetilde{\bL}$ are non-negative. We will denote by $q$ the rank of $\widetilde{\bL}$. Note that $q \le n-p$ as the $p$ columns of $\bQ$ are trivially eigenvectors of $\widetilde{\bL}$ associated to $0$. We write
 		\[
 		\widetilde{\bL} = \tbU\tbLam\tbU^\top
 		\] 
 		its truncated spectral decomposition; where $\tbLam=\text{diag}(\widetilde{\lambda}_1,\ldots,\widetilde{\lambda}_q)\in \mathbb{R}^{q\times q}$ and $\tbU \in \mathbb{R}^{n\times q}$ are the diagonal matrix of nonzero eigenvalues and the matrix of the corresponding eigenvectors of $\widetilde{\bL}$, respectively.
 	\end{itemize}
 \end{definition}
 \begin{remark}
 	Note that we authorize $p=0$ in the definition: in this case, $\bQ=0$ and $\tbL=\bL$.
 \end{remark}
 
 \begin{definition}[Extended L-ensemble] 
 	\label{def:ext_Lens}
 	Let $\ELE{\bL}{\bV}$ be any NNP. An extended L-ensemble $\X$ based on $\ELE{\bL}{\bV}$ is a point process verifying:
 	\begin{align}
 	\label{eq:proba_mass}
 	\forall X\subseteq\Omega,\qquad \Proba(\X=X) = \frac{1}{Z} \det
 	\begin{pmatrix}
 	\bL_{X} & \bV_{X,:} \\
 	(\bV_{X,:})^\top & \matr{0}
 	\end{pmatrix},
 	\end{align} 
 	with $Z$ the normalization constant verifying:
 	$$Z=\sum_{X} \det
 	\begin{pmatrix}
 	\bL_{X} & \bV_{X,:} \\
 	(\bV_{X,:})^\top & \matr{0} 
 	\end{pmatrix} = (-1)^p\det(\bI+\tbL)\det(\bV^\top\bV).$$
 \end{definition}

 \begin{remark}
	If $p=0$, then the extended L-ensemble $\X$ based on $\ELE{\bL}{\bV}$ reduces to a simple L-ensemble associated to the semi-positive definite matrix $\bL$. In particular, the probability of sampling the empty set, $1/\det(\bI+\bL)$, is always positive. \\
	If $p>0$, then the size of the extended L-ensemble $\X$ based on $\ELE{\bL}{\bV}$, denoted by $|\X|$, is necessarily superior or equal to $p$ (in particular, $\X$ can never be the empty set); and we call such DPPs \emph{partial projection DPPs} (pp-DPPs) for reasons that become clear when one studies their mixture representation (the projective part comes from $\bV$)~\cite[Section 3.2]{paper1}
\end{remark}
 Importantly, the class of extended L-ensembles is identical to the class of DPPs:
 
 \begin{theorem}{Thm. 2.9 in\cite{paper1}}
 	\label{thm:K_to_L-ens}
 	i/~Let $\ELE{\bL}{\bV}$ be any NNP, and $\X$ be an extended L-ensemble based on $\ELE{\bL}{\bV}$. 
 	Then, $\X$ is a DPP with marginal kernel
 	\begin{equation}
 	\label{eq:dpDPPmarginalkernel}
 	\bK = \bQ \bQ^\top + \tbL(\bI+ \tbL)^{-1}.
 	\end{equation}
 	ii/~Let $\bm{0}\preceq\bK\preceq \bI$ be any marginal kernel and $\X$ its associated DPP. Denote by $\bV\in\mathbb{R}^{n\times p}$ the matrix concatenating the $p\geq 0$ orthonormal eigenvectors of $\bK$ associated with eigenvalue $1$ and $\bL=\bK\left(\bI-\bK\right)^{\dagger}$ with $\dagger$ representing the Moore-Penrose pseudo-inverse. Then, $\X$ is an extended L-ensemble based on the NNP $\ELE{\bL}{\bV}$.
 \end{theorem}
\begin{remark}[Notation] As a consequence, we denote the extended L-ensemble based on the NNP $\ELE{\bL}{\bV}$ by  $\X\sim \ppDPP \ELE{\bL}{\bV}$. In the specific case where $p=0$ for which the underlying process reduces to a simple L-ensemble, we will simply write $\X\sim DPP(\bL)$.
\end{remark}
 
Thus, any DPP $\mathcal{X}$ may equivalently be defined via 
 	\begin{itemize}
 		\item its marginal probabilities. This requires the definition of a marginal kernel $\bK$ verifying $\bm{0}\preceq \bK \preceq \bm{I}$ and all marginals at any order are given by Eq.~\eqref{eq:def-dpp}.
 		\item its joint probability distribution. This requires the definition of a NNP $\ELE{\bL}{\bV}$ as in Definition~\ref{def:nnp}, and the probability distribution is given by Eq.~\eqref{eq:proba_mass}.
 \end{itemize}
 
 \subsection{The fixed-size case}
 In general, the size of a DPP sample is itself random. In many practical scenarios however, one wishes to control the size of the random sample, which lead authors in~\cite{KuleszaTaskar:FixedSizeDPPs} to introduce fixed-size DPPs:
 \begin{definition}[Fixed-size Determinantal Point Process]
 	\label{def:fsDPP}
 	A fixed size DPP of size $m$ is a DPP $\X$ conditioned on $|\X|=m$. 
 \end{definition}

\begin{definition}[Fixed-size L-ensemble]
	\label{lemma:mDPP_via_L-ens}
	Let $\bm{0}\preceq\bL$ be a positive semi-definite matrix. A fixed-size L-ensemble is a point process $\X$ defined as:
	\begin{align}
	\label{eq:proba_mass_Lens}
	\forall X\subseteq\Omega,\qquad \Proba(\X=X) = \frac{\det
		\bL_{X}}{e_{m}(\bL)} \;\Ind(|X|=m).
	\end{align}
	where $e_m(\bL)$ is the $m$-th elementary symmetric polynomial and $\Ind(\cdot)$ is the indicator function.
\end{definition}

Again, fixed-size L-ensembles are fixed-size DPPs, but the contrary is false. The extended L-ensemble representation enables to fill in that gap.  
Indeed, as a consequence of the equivalence between extended L-ensembles and DPPs, one obtains:
 \begin{corollary}
 	Let $\bm{0}\preceq\bK\preceq \bI$ be any marginal kernel and $\X$ its associated fixed-size DPP of size $m$. Then, $\X$ is a fixed-size extended L-ensemble associated to the NNP $\ELE{\bL}{\bV}$ as defined in theorem~\ref{thm:K_to_L-ens}. Moreover:
 	\begin{align}
 	\label{eq:proba_mass_fs}
 	\forall X\subseteq\Omega,\qquad \Proba(\X=X) = \frac{1}{Z_m} \det
 	\begin{pmatrix}
 	\bL_{X} & \bV_{X,:} \\
 	(\bV_{X,:})^\top & \matr{0}
 	\end{pmatrix}\Ind(|X|=m)
 	\end{align}
 	where $Z_m $ is the normalization constant verifying:
 	$$Z_m = \sum_{|X|=m} \det
 	\begin{pmatrix}
 	\bL_{X} & \bV_{X,:} \\
 	(\bV_{X,:})^\top & \matr{0} 
 	\end{pmatrix} = (-1)^p \;e_{m-p}(\tbL)\det(\bV^\top\bV)$$
 	and $e_k(\tbL)$ is the $k$-th elementary symmetric polynomial.
 \end{corollary}

 \begin{remark}
	If $p=0$, then the fixed-size extended L-ensemble $\X$ based on $\ELE{\bL}{\bV}$ reduces to a simple fixed-size L-ensemble associated to $\bL$.
	If $p>0$, the associated fixed-size DPP is called a fixed-size partial projection DPPs (fixed-size pp-DPPs); and it can only be defined for $m\geq p$. 
\end{remark}

\begin{remark}[Notation] We denote the fixed-size extended L-ensemble based on the NNP $\ELE{\bL}{\bV}$ by  $\X\sim \mppDPP{m} \ELE{\bL}{\bV}$. In the specific case where $p=0$ for which the process reduces to a fixed-size L-ensemble, we will simply write $\X\sim \mDPP{m} (\bL)$.
\end{remark}

Note that a fixed-size DPP is not a DPP in general, with the notable exception of projection DPPs:
\begin{definition}[Projection DPP]
	Let $\bU$ be an $n \times m$ matrix with $\bU^\top \bU = \bI_m$. A projection DPP $\X$ is a DPP with marginal kernel $\bK=\bU\bU^\top$. Equivalently (see e.g., {\cite[Lemma 1.3]{Barthelme:AsEqFixedSizeDPP}}), it can be described as a fixed-size L-ensemble $\X \sim \mDPP{m}(\bU\bU^\top)$.
\end{definition}

In the remainder of the paper, we will need the following lemma, which gives
another characterisation of projection DPPs: 
\begin{lemma}[See e.g., {\cite[Lemma 1.25]{paper1}}]
	\label{lem:max-rank-dpp}
	Let $\X \sim \mDPP{m}(\bL)$ with $\rank(\bL) = m$, and let $\bU\in \R^{n \times m}$ denote an
	orthonormal basis for $\mspan \bL$.  Then, equivalently, $\X \sim
	\mDPP{m}(\bU \bU^\top)$.
\end{lemma}

\section{The flat limit of fixed-size L-ensembles (univariate case)}
\label{sec:univariate-results}

In this section and the two following ones, we study L-ensembles based on kernel matrices taken in the flat limit. 
Our two main general theorems, one for the fixed-size case and one for the much
more involved varying-size case, are in Section~\ref{sec:two_general_theorems}.
We are well aware that they are technical and, in order to guide the reader into
those results, we propose, in Section~\ref{sec:univariate-results}, to start
gently with corollaries stating the limits of \emph{fixed-size L-ensembles} in
the \emph{univariate} case (the ground set $\Omega$ is a subset of the real line). Then, Section~\ref{sec:results-multivariate} presents corollaries stating the limits obtained  in the \emph{multivariate} case ($\Omega\subseteq\mathbb{R}^d$, $d\geq1$), but still in the fixed-size context. Both sections do not contain proofs (they are all in Section~\ref{sec:two_general_theorems}) and are devoted to provide intuitions on our main results, via examples and numerical illustrations. 

This section is organized as follows: we begin by defining our objects of study as well as the necessary notation. We then give corollaries stating the limits of \emph{fixed-size L-ensembles} in the \emph{univariate}  case which, as we will see, depend mostly on $r$, the smoothness parameter of the kernel. The section concludes with some numerical results.

\subsection{Introduction}
\label{sec:1d-intro}

We focus on stationary kernels, as defined in section
\ref{sec:intro-kernels}, where $\varepsilon$ plays the role of an inverse scale
parameter. Thus, we consider L-ensembles based on matrices of the form
\[ \bL(\varepsilon) = [\kappa_\varepsilon(x_i,x_j)]_{i=1,j=1}^n\]
for a set of points $\Omega = \{x_1, \ldots, x_n\}$, all on the real line and all different from one another. 
From stationarity, the kernel function $\kappa_\varepsilon$ may be
written as:
\[ \kappa_\varepsilon(x_i,x_j) = f(\varepsilon|x_i-x_j|)\]
and we further assume that $f$ is analytic in a neighbourhood of 0. As in
equation \eqref{eq:kernel-expansion}, we expand the kernel in powers of $\varepsilon$ as:
\[ \kappa_\varepsilon(x_i,x_j) = f_0 +  \varepsilon f_1 |x_i-x_j| +
  \varepsilon^2 f_2 |x_i-x_j|^2  + \varepsilon^3 f_3|x_i-x_j|^3 +  \ldots\]
The expansion for individual entries may be represented in a more compact and
familiar manner in a matrix form:
\begin{equation}
  \label{eq:kernel-matrix-expansion}
  \bL(\varepsilon) = f_0 \bD^{(0)} + \varepsilon f_1 \bD^{(1)} + \varepsilon^2 f_2 \bD^{(2)} + \ldots
\end{equation}
where 
\[ \bD^{(p)} = [|x_i-x_j|^p]_{i,j}\]
Our goal is to characterise the limiting processes that arise from fixed-size (and later varying-size in Section~\ref{sec:two_general_theorems}) L-ensembles based on $\bL(\varepsilon)$ as $\flatlim$. It is 
useful to think of the terms $\varepsilon^p f_p \bD^{(p)}$ as containing
features that are increasingly down-weighted as $\flatlim$. The analysis is notably
complicated by the fact that the matrices
$\bD^{(p)}$ are rank-deficient for even $p$ (up to some index depending on $n$) 
but invertible for odd $p$ \cite{BarthelmeUsevich:KernelsFlatLimit}. 
The smoothness order of the kernel (see section
\ref{sec:intro-kernels}) defines how soon in the decomposition the first
invertible matrix appears. For instance, if $r=2$ then $f_1 = 0$ and we get:
\[   \bL(\varepsilon) = f_0 \bD^{(0)} +
  \varepsilon^2 f_2 \bD^{(2)} + \varepsilon^3 f_3 \bD^{(3)} + \ldots  \]
If $n>2$, the first invertible matrix to appear in the expansion in
$\varepsilon$ is $\bD^{(3)}$, and it will lead to different asymptotic
behaviour than if the first invertible matrix had been $\bD^{(1)}$ ($r=1$) or
$\bD^{(5)}$ ($r=3$). If the kernel is completely smooth, then:
\[   \bL(\varepsilon) = \sum_{p=0}^\infty \varepsilon^{2p}f_{2p}\bD^{(2p)}  \]
and odd terms never appear. This again has its own asymptotic behaviour. A subtle issue is that if the matrix under
consideration is small enough compared to the regularity order, then the
asymptotics are the same than in the completely smooth case. We invite the
reader to pay attention to the interplay between $m$ (the size of the L-ensemble) and
$r$ (the regularity order) in the following results. For more on the flat asymptotics of kernel matrices, we refer again to \cite{BarthelmeUsevich:KernelsFlatLimit}. 

\subsubsection{Vandermonde matrices}
\label{sec:univar-poly}

 The Vandermonde matrix of order $k$ is defined as:
\begin{equation}\label{eq:Vandermonde1D}
  \matr{V}_{\le k} = \begin{bmatrix}
    1 & x_1 & \cdots & x^{k}_1\\
    \vdots &  & \vdots \\
    1 & x_n & \cdots & x^{k}_n 
  \end{bmatrix},
\end{equation}
where $x_1, \ldots, x_n$ are the $n$ points of the ground set $\Omega$ (we also use
the notation $\bV_{<k} = \bV_{\leq k-1}$).
Note that $\matr{V}_{\le k}$ has $k+1$ columns. 
The ``classical'' Vandermonde matrix is obtained for $k=n-1$, which makes it square.
$\bV_{\le n-1}$ is invertible if and only if the points in $\Omega$ are distinct, which can
be established from the following well-known determinantal formula:
\begin{equation}
  \label{eq:VdM-determinant}
  \det \matr{V}_{\le n-1} = \prod_{i<j} (x_i-x_j)
\end{equation}
As short-hand, we shall define $\vect{v}_l=
\begin{pmatrix}
  x_1^{l}, \ldots, x_n^{l}
\end{pmatrix}^\top\in\mathbb{R}^{n}$, such that \begin{equation}\label{eq:Vandermonde1D}
\matr{V}_{\le k} = \left[
\vect{v}_0 | \vect{v}_1 | \ldots | \vect{v}_{k}
\right].
\end{equation}
Submatrices of $\bV_{\leq k}$ corresponding to a subset of points $X$ will be
denoted $\bV_{\leq k}(X) \in \R^{|X| \times (k+1) }$.


\subsection{The flat limit in the fixed-size case}
\label{sec:univariate-fixed-size}

Consider $\X_\varepsilon \sim \mDPP{m}(\bL(\varepsilon))$ with $m\leq n$ and $m$ and $n=|\Omega|$ fixed (no large $n$ asymptotics are involved here). We are interested in
the limiting distribution of $\X_\varepsilon$ as $\flatlim$.

It is not at first blush obvious that the limiting point process exists
and is non-trivial. Indeed, as $\varepsilon\rightarrow0$, every entry of the
matrix $\bL(\varepsilon)$ goes to 1, and so $\det(\bL(\varepsilon)_X)$ goes to 0 for all
subsets $X$. What makes the limit non-trivial is, as we shall see in the
proofs, that these quantities go to 0 at different speeds.

The first result characterises the smooth case, where the smoothness order of
the kernel is larger than $m$ (recall that this applies to the Gaussian kernel,
for instance)
\begin{corollary}
  \label{cor:cs-case-fixed-size-1d}
  Let $\bL_\varepsilon = [\kappa_\varepsilon(x_i,x_j)]_{i,j}$ with $\kappa$ a
  stationary kernel of smoothness order $r \geq m$. Then $\Xe \sim
  \mDPP{m}(\bL_\varepsilon)$ converges to $\Xs \sim \mDPP{m}(\bV_{< m}
  \bV_{< m}^\top)$.   
\end{corollary}
\begin{proof}
Subcase of Theorem~\ref{thm:general-case-smooth-fixed-size}.
\end{proof}
\begin{remark}
  The result says that as $\flatlim$ the limiting point process is (a) a fixed-size L-ensemble\footnote{and even a projection DPP as $\bV_{< m} \bV_{< m}^\top$ is of rank $m$ (see Lemma~\ref{lem:max-rank-dpp})}
  and (b) the positive semi-definite matrix it is based on is a Vandermonde matrix of
  $\Omega$. It is worth studying this matrix in greater detail. Let
  $\matr{M} = \bV_{< m} \bV_{< m}^\top$. Then for any subset $X \subset
  \Omega$ of size $m$, $\det \matr{M}_X=\det^2(\bV_{< m}(X))$, because
  $\bV_{< m}(X)$ is a square matrix. From the Vandermonde determinant
  formula (eq. \eqref{eq:VdM-determinant}), this means that if $\X \sim \mDPP{m}(\matr{M})$,
  \begin{equation}
    \label{eq:cs-case-joint-prob}
    \Proba\left(\X=X\right) = \frac{1}{Z} \prod_{(x,y)\in X^2} (x - y)^2
\end{equation}
  
\end{remark}
\begin{remark}
  Consider the conditional inclusion probability for a single point $\Proba(\Xe = \{x\} \cup Y |
  Y)$, to be read as the conditional
  probability that $\Xe = \{ x\} \cup Y$ given that $Y \subseteq X_\epsilon$. In
  the flat limit, this quantity tends to:
  \[ \Proba\left(\Xs = \{x\} \cup Y | Y\right) \propto \prod_{y \in Y} (x-y)^2\]
  which corresponds to a repulsive point process (since small distances between
  points are unlikely). 
\end{remark}
To summarise: if we sample a fixed-size L-ensemble of size $m$, and the kernel is
regular enough compared to $m$ (\emph{i.e.}, $r\geq m$), then \emph{whatever} the kernel the
limiting process exists and is the same\footnote{The ``whatever the kernel''
  part becomes more complicated in the multidimensional case, as we shall see.}. The probability of sampling a set $X$
is just proportional to a squared Vandermonde determinant, and that defines a projection DPP.

The next theorem describes what happens when the kernel is less smooth. We obtain a partial projection
DPP, where the projective part comes from polynomials, and the non-projective
part comes from the first nonzero odd term in the kernel expansion (see Eq.~\eqref{eq:kernel-matrix-expansion}).

\begin{corollary}
  \label{cor:fs-case-fixed-size-1d}
  Let $\bL_\varepsilon = [\kappa_\varepsilon(x_i,x_j)]_{i,j}$ with $\kappa$ a
  stationary kernel of smoothness order $r \leq m$. 
  Then $\Xe \sim
  \mDPP{m}(\bL_\varepsilon)$ converges to $\Xs \sim \mppDPP{m} \ELE{\bD^{(2r-1)}} {\bV_{< r}}$.   
\end{corollary}
\begin{proof}
 Subcase of Theorem~\ref{thm:general-case-smooth-fixed-size}.
\end{proof}
\begin{example}
  In the case of the exponential kernel $\kappa_{\varepsilon}(x,y) = e^{ -\varepsilon |x-y| }$, $r=1$, and the theorem states:
  \begin{equation}
    \label{eq:extended-L-ensemble-exp-kernel}
    \Proba(\Xs = X) \propto \det
    \begin{pmatrix}
      -\bD_{X}^{(1)} & \ones \\
      \ones^\top & 0
    \end{pmatrix}
  \end{equation}
\end{example}

\begin{remark}
  \label{rem:exp-kernel-d1}
  Some algebra reveals that
  \begin{equation}
    \label{eq:det-exp-kernel}
    \det
    \begin{pmatrix}
      -\bD_{X}^{(1)} & \ones \\
      \ones^\top & 0
    \end{pmatrix}
    = 2^{m-1} \prod_{i = 1}^m (x_{i+1} - x_i)
  \end{equation}
  where in the last expression we have sorted the points in $X$ so that $x_1
  \leq x_2 \leq \ldots \leq x_m$. 
  As in  \eqref{eq:cs-case-joint-prob} above, the repulsive nature of the limit
  point process is immediately apparent from eq. \eqref{eq:det-exp-kernel}.
  Unlike \eqref{eq:cs-case-joint-prob}, which involves all distances, eq.
  \eqref{eq:det-exp-kernel} only involves distances between direct neighbours.
  We speculate that similar expressions exist for $r>1$, but have unfortunately
  been unable to derive them.
\end{remark}
\begin{proof}
  Eq. \eqref{eq:det-exp-kernel} may be derived by using a finite
  difference operator of the form:
  \[
    \matr{F} =
    \begin{pmatrix}
      1 & 0 & \ldots \\
      \frac{-1}{\delta_1} & \frac{1}{\delta_1} & 0 & \ldots \\
      0 & \frac{-1}{\delta_2} & \frac{1}{\delta_2} & 0 & \ldots \\
      \vdots & \vdots & \vdots & \vdots
    \end{pmatrix}
  \]
  where $\delta_i = x_{i+1}-x_i$. Since $\matr{F}$ is lower-triangular, $\det
  \matr{F}= \prod_{i=1}^{m-1} \delta_i^{-1}$.  Then applying
  lemma \ref{lem:det-saddlepoint} to
  \[ \det(
    \begin{bmatrix}
      \matr{F} & 0 \\
      0 & 1
    \end{bmatrix}
    \begin{bmatrix}
      -\bD_{X}^{(1)} & \ones \\
      \ones^\top & 0
    \end{bmatrix}
    \begin{bmatrix}
      \matr{F}^\top & 0 \\
      0 & 1
    \end{bmatrix}
    ) \] and simplifying yields the result.
\end{proof}

\subsection{Some numerical illustrations}
\label{sec:numerics-1d}

To illustrate the convergence theorems above, a good visual tool is to examine
the convergence of conditional distributions of the form:
\begin{equation}
  \label{eq:cond-law-dpp}
  \Proba(\X = \{x\} \cup Y | Y) \propto \det \bL_{\{x\} \cup Y} \propto  (L_{x,x}-\bL_{x,Y}\bL_{Y}^{-1} \bL_{Y,x})
\end{equation}

This should be interpreted as the conditional probability of the $m$-th item
fixing the first $m-1$. The conditional law $\Proba(\Xe = \{x\} \cup Y | Y)$ tends
to that of $\Proba(\Xs = \{x\} \cup Y | Y)$, and in dimension 1 we can depict this 
 as a function of $x$.

We do so in figure \ref{fig:convergence-cond-1d}, where we assume $\X$ is a
$m=5$ fixed-size L-ensemble, and the ground set is a finite subset of $[0,1]$. The conditioning subset $Y$
is chosen to be of size 4, 
and for the sake of illustration, we let $x$ vary as a continuous parameter in $[0,1]$.
 The four panels correspond to four different kernel
functions. The conditional probability is plotted for different values of $\varepsilon$. In all plots we observe a rapid convergence
with $\varepsilon$. In the top panel, the difference between the asymptoptics obtained for $r=1$ and
$r = \infty$ are quite striking. In the bottom panel, we have two different
kernels with identical smoothness index, and as predicted by Corollary
\ref{cor:cs-case-fixed-size-1d} the $\flatlim$ limits are identical.
 
\begin{figure}
  
  \centering
  \begin{subfigure}[b]{7.0cm}
    \includegraphics[width=7.0cm]{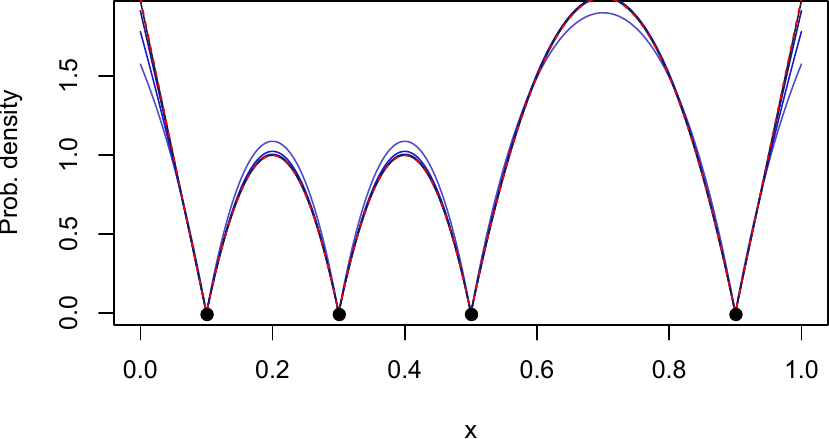}
    \caption{$k(x,y)=\exp(-|x-y|)$, a kernel with $r=1$}
  \end{subfigure}
  \begin{subfigure}[b]{7.0cm}
    \includegraphics[width=7.0cm]{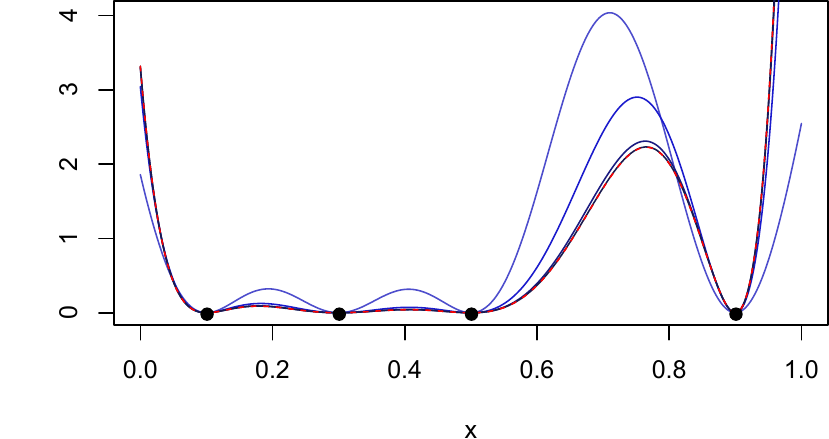}
    \caption{$k(x,y)=\exp(-(x-y)^2)$,  $r=\infty$}
  \end{subfigure}
  \\
  \begin{subfigure}[b]{7.0cm}
    \includegraphics[width=7.0cm]{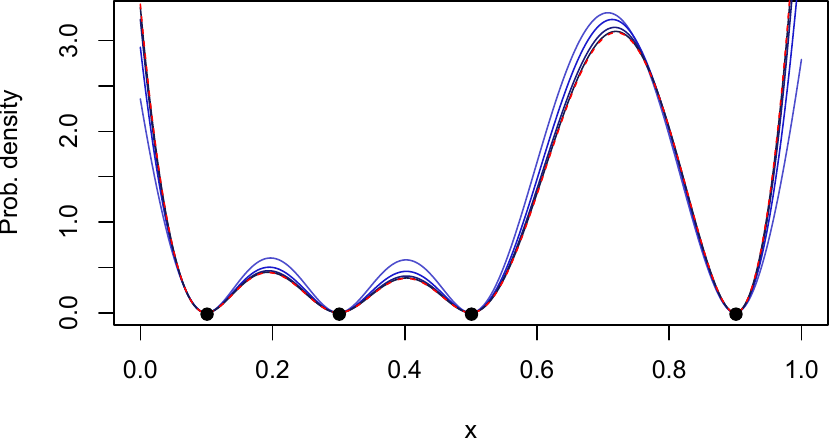}
    \caption{$k(x,y)=(1+|x-y|)\exp(|x-y|)$,  $r=2$}
  \end{subfigure}
  \begin{subfigure}[b]{7.0cm}
    \includegraphics[width=7.0cm]{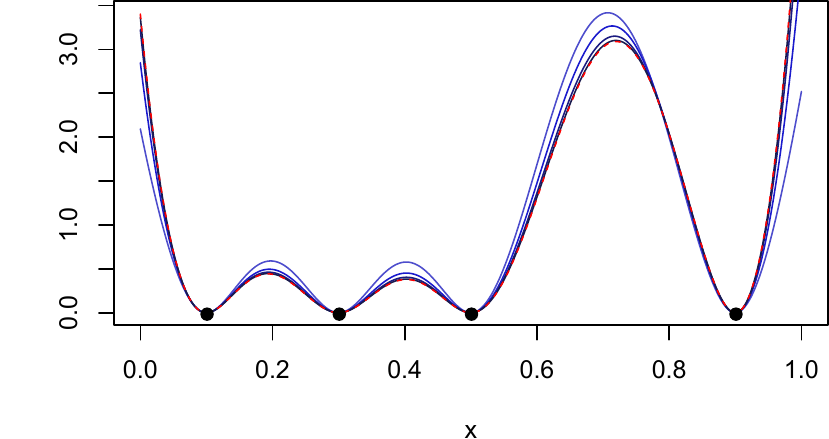}
    \caption{$k(x,y)=\sin(|x-y|+\frac{\pi}{4})\exp(-|x-y|)$,  $r=2$}
  \end{subfigure}

  \caption{Asymptotics of conditional densities of fixed-size L-ensembles based on four different
    kernels. Here we plot $\Proba(\Xe = \{x\} \cup Y | Y)$, the conditional density of a fixed size L-ensemble
    (with $m=5$) where four of the points are fixed ($Y$) and the last is varying
    $(x)$. The points in $Y$ are at $0.1,0.3,0.5,0.9$. The curves in blue are
    the conditional densities for different values of $\varepsilon$:
    $4,1.5,.5,.1$. The dotted red line is
  the asymptotic limit in $\flatlim$. Note that the two kernels in the bottom row have
the same regularity coefficient $r=2$, and as predicted by the results the
limiting densities are equal.}
\label{fig:convergence-cond-1d}

\end{figure}

Another set of quantities that are easy to examine visually are the first order
inclusion probabilities ($\Proba(x \in \X)$). We refer to
\cite{Barthelme:AsEqFixedSizeDPP} for how to compute these quantities in
fixed-size L-ensembles. Since $\Xe$ converges to $\Xs$, so must the inclusion
probabilities, and this is shown in figure \ref{fig:convergence-inclusion-1d} for
three kernels with increasing values of $r$.  For these plots, the ground set consists in 20 points drawn at random in the unit interval.
We depict the first order inclusion probabilities for four different values of $\varepsilon$. Rapid convergence with $\varepsilon$ is also observed. 
\begin{figure}
  \includegraphics[width=12cm]{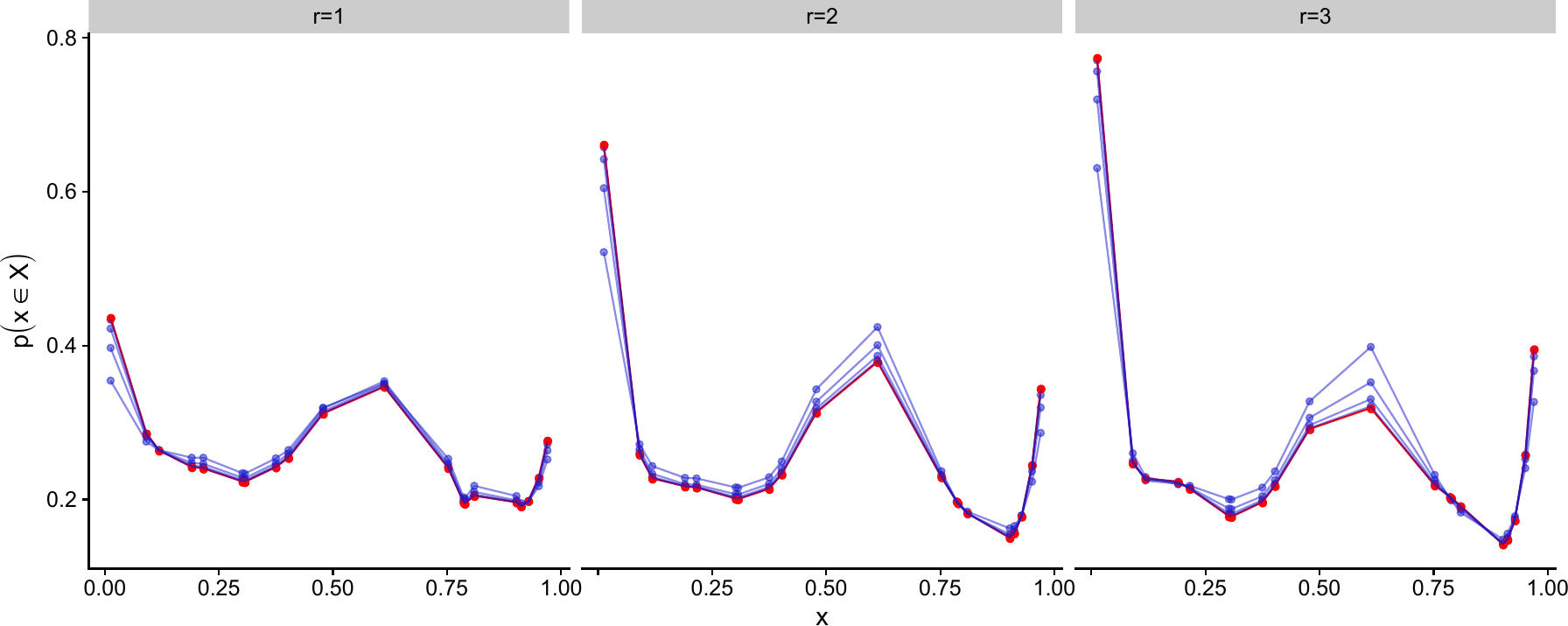}
  \caption{
    Flat limit of inclusion probabilities of (fixed-size) L-ensembles for three different
    kernels. Here we plot $\Proba(x \in \Xe)$, the inclusion probabilities for a fixed size L-ensemble
    (with $m=5$), where the ground set $\Omega$ consists in 20 points drawn at
    random from the unit interval. The dots in blue (joined by lines for
    clarity) are inclusion probabilities for 
    $\varepsilon = 4,1.5,.5,.1$. The dots in red correspond to
  the asymptotic limit in $\flatlim$. The three kernels are, from left-to-right,
$\exp(-\delta), (1+\delta)\exp(-\delta),(3+3\delta+\delta^2)\exp(-\delta)$,
where $\delta = |x - y|$. These kernels have $r=1$, $2$ and $3$, respectively.}
\label{fig:convergence-inclusion-1d}
\end{figure}

\section{The flat limit of fixed-size L-ensembles (multivariate case)}
\label{sec:results-multivariate}

The univariate results stated above have a multivariate generalisation, and
in some cases they are almost the same. The only major difference is that in the
univariate case, the \emph{only} aspect of the kernel function that plays a role
in determining the limiting process is the smoothness order $r$. Two kernels may
look different, but if they have the same smoothness order they have the same
limiting DPP. 
When $d>1$ this is no longer always true. The limiting process may \emph{sometimes}
depend on the specific values of the derivatives of the kernel at 0 (not just whether they exist). 
Sometimes, but not always: for
instance, all kernels with $r=1$ give the same limiting fixed-size DPP. All kernels
with $r=2$ give the same limiting fixed-size ($m$) L-ensemble, as long as $m>d$. The case of
infinitely smooth kernels is particularly intriguing: there is a universal
limiting process, but only for $m$ in a set of ``magic'' values $\magicn{d}$ to be defined below. 
 When $m$ falls in between these values, then the
limiting process depends on the kernel (although perhaps not strongly).

To build a picture of what the final results look like, we state the easiest first:
\begin{example}
  Let $d>0$ and $\vect{x}_1, \ldots, \vect{x}_n$ be $n$ points in dimension $d$. Let $\bL_\varepsilon = [\kappa_\varepsilon(\vect{x}_i,\vect{x}_j)]_{i,j}$ with
  $\kappa$ a stationary kernel of smoothness order $r = 1$. 
    Then $\Xe \sim  \mDPP{m}(\bL_\varepsilon)$ converges to $\Xs \sim
  \mppDPP{m}  \ELE{-\bD}{\ones} $, where $\bD$ is the distance matrix $\bD = \left[
  \norm{\vect{x}_i-\vect{x}_j}\right]_{i,j}$. 
\end{example}
A more general statement is given later, but this one has the advantage of being
identical to the univariate result. 

As the more general statements are also more complicated, we present our
results in increasing order of complexity. The general theorem is Theorem~\ref{thm:general-case-smooth-fixed-size}, and all results we state in this section (including the above) are
special cases. But before delving into this, we need to recall some aspects of Vandermonde matrices in higher dimensions and introduce the magic numbers $\magicn{d}$. 

\subsection{Multivariate polynomials, multivariate Vandermonde matrices}~
\label{sec:multivar-poly} 

\noindent\textbf{Multivariate polynomials.} We
recall here the essential facts on multivariate polynomials. 
Let $\vect{x} =
\begin{pmatrix}
x_1 & x_2 & \ldots & x_d
\end{pmatrix}^\top
\in \R^d $. A monomial in $\vect{x}$ is a function of
the form:
\[ \vect{x}^{\vect{\alpha}} = \prod_{i=1}^d x_i^{\alpha_i}  \]
for $\vect{\alpha} \in \mathbb{N}^d$ (a multi-index). Its total degree (or degree for short) is defined as
$|\vect{\alpha}|=\sum_{i=1}^d \alpha_i$. For instance:
\[ \vect{x}^{(2,1)} = x_1^2x_2 \]
and it has degree 3. A multivariate polynomial in $\vect{x}$ is a weighted sum
of monomials in $\vect{x}$, and its degree is equal to the maximum of the
degrees of its component monomials. For instance, the following is a
multivariate polynomial of degree 2 in $\R^3$:
\[ \vect{x}^{(0,1,1)}-\vect{x}^{(1,0,1)}+2.2\vect{x}^{(1,0,0)}-1.\]
One salient difference between the univariate and the multivariate case is that
when $d>1$, there are several monomials of any given degree, instead of just
one. For instance, with $d=2$, the first few monomials are (by increasing
degree):
\begin{align*}
\vect{x}^{(0,0)} \\
\vect{x}^{(1,0)},\vect{x}^{(0,1)} \\
\vect{x}^{(2,0)},\vect{x}^{(1,1)},\vect{x}^{(0,2)}\\
\end{align*}
There is a well-known formula for counting monomials of degree $k$ in dimension $d$:
\begin{equation}
\label{eq:number-monomials}
\HH_{k,d} = {k+d-1 \choose d-1}.
\end{equation}
The notation $\HH_{k,d}$ comes from the notion of \emph{homogeneous polynomials}. A
homogeneous polynomial is a polynomial made up of monomials with equal degree.
Therefore, the set of homogeneous polynomials of degree $k$ has dimension
$\HH_{k,d}$. The set of polynomials of degree $k$ is spanned by the sets of
homogenous polynomials up to $k$, and has dimension:
\begin{equation}
\label{eq:dim-polynomials}
\PP_{k,d} = \HH_{0,d} + \HH_{1,d} + \ldots + \HH_{k,d} = {k+d \choose d}.
\end{equation}
Note for instance that $\PP_{0,d}=1$ and $\PP_{1,d}=d+1$. Note also that when the dimension equals $1$, one recovers that the set of univariate polynomials of degree $k$ has dimension $k+1$. By convention, we will set $\PP_{-1,d}$ to be equal to $0$.\\

\noindent\textbf{Multivariate Vandermonde matrices.} We  now  define the multivariate generalisation of Vandermonde
matrices. Monomials are naturally ordered by degree, but monomials of the same
degree have no natural ordering. To properly define our matrices, we require
(formally) an ordering. For the purposes of this paper which ordering is used is
entirely arbitrary. For more on orderings, see
\cite{BarthelmeUsevich:KernelsFlatLimit} and references therein.
For an ordered set of points $\Omega = \{\vect{x}_1, \ldots, \vect{x}_n\}$, all in $\RR^d$, we define the multivariate Vandermonde matrix as:
\begin{equation}\label{eq:VandermondeND}
\matr{V}_{\le k} = 
\begin{bmatrix}
\matr{V}_{0} & \matr{V}_{1}  & \cdots & \matr{V}_{k}  
\end{bmatrix} \in \mathbb{R}^{n\times \PP_{k,d}},
\end{equation}
where each block $\matr{V}_i \in\mathbb{R}^{n\times \HH_{i,d}}$ contains the monomials of degree $i$ evaluated on
the points in $\Omega$ (we will also use the notation $\matr{V}_{< k}=\matr{V}_{\le k-1}$). As an example, consider $n=3$, $d=2$ and the ground set
\[
\Omega = \{\left[\begin{smallmatrix}y_1 \\z_1 \end{smallmatrix}\right], \left[\begin{smallmatrix}y_2 \\z_2 \end{smallmatrix}\right], \left[\begin{smallmatrix}y_3 \\z_3 \end{smallmatrix}\right]\}.
\]
One has, for instance for $k=2$:
\[
\matr{V}_{\le 2} =
\left[\begin{array}{c|cc|ccc}
1 & y_1 & z_1 & y_1^2 & y_1z_1 & z_1^2 \\
1 & y_2 & z_2 & y_2^2 & y_2z_2 & z_2^2 \\
1 & y_3 & z_3 & y_3^2 & y_3z_3 & z_3^2 \\
\end{array}\right],
\]
where the ordering within each block is arbitrary. 

As in the previous section, we use $\bV_{\le k}(X)$ to denote the matrix $\bV_{\le k}$ reduced to its lines indexed by the elements in $X$. As such, $\bV_{\le k}(X)$ has $|X| = m$ rows and $\PP_{k,d} $ columns. For some values of $m$ and $k$ it is square and (potentially) invertible. For instance, consider $\bV_{\leq k}$ as in Eq.~\eqref{eq:VandermondeND}, with $k=1$ and $d=2$. Choosing a subset $X$ of size $m=3$, the matrix $\bV_{\leq 1}(X)$ is square. In dimension 2, there exists a square Vandermonde matrix for sets $X$ of size $m=1$, $3$, $6$, $10$, $15$, $21$, etc.
	
In fact, for any dimension $d$, there exists a square Vandermonde matrix for any size $m$ such that there exists $k\in\mathbb{N}$ verifying $\PP_{k,d} =m$, that is, any $m$ included in the set of integers:
\begin{equation}
\label{eq:magic-numbers}
\magicn{d} = \left\{ \PP_{k,d} | k \in \mathbb{N}\right\}.
\end{equation}
We will see that these values of $m$ are in some sense natural sizes for L-ensembles, because they
lead to universal limits, and that is the reason for calling them magic numbers. 

We note in passing that while we may easily determine whether $\bV_{\leq k}(X)$ is
square, whether it is invertible is a complicated question that depends on the
geometry of the points $X$, as there are some non-trivial configurations for
which it is not \cite{gasca2000polynomial}. The results below show that such configurations
 have probability 0 in the flat limit under any L-ensemble with $r$ sufficiently large compared to $m$.  \\

\subsection{Universal (easy) limits}
\label{sec:universal-limits}

The following result applies when the kernel is sufficiently smooth and the L-ensemble has
fixed size $m \in \magicn{d}$.
\begin{corollary}
  \label{cor:nd-smooth-magic-case}
  Let $d\in \mathbb{N}^*$ and $\bL_\varepsilon =
  [\kappa_\varepsilon(\vect{x}_i,\vect{x}_j)]_{i,j}$ for $\kappa$ a stationary
  kernel of smoothness order $r$ and $\vect{x}_1, \ldots, \vect{x}_n$
  vectors in $\R^d$. Then for all $m\in \{\PP_{k,d}\}_{k\leq r-1} \subset\magicn{d}$, the fixed-size L-ensemble $\X_\varepsilon \sim
  |DPP|_m(\bL(\varepsilon))$ has the limiting distribution:
  \[ \Xs \sim |DPP|_m(\bV_{\leq k}\bV_{\leq k}^\top)\]
  Equivalently, if $\bQ$ is an orthonormal basis for $\bV_{\leq k}$, then:
  \[ \Xs \sim |DPP|_m(\bQ\bQ^\top)\]
\end{corollary}
\begin{proof}
 Subcase of Theorem~\ref{thm:general-case-smooth-fixed-size}.
\end{proof}
\begin{remark}
  Since $\bV_{\leq k}$ is a polynomial basis, $\bQ$ is a basis of orthogonal
  polynomials. The limiting process we see appearing here is the same as the one
  studied in \cite{tremblay2019determinantal} in the discrete case. A similar
  theorem can be proved for continuous DPPs, essentially by tediously changing
  the notation, and leads to the multivariate orthogonal ensembles studied in
  \cite{bardenet2020monte}. What this means is that the properties proved in
   \cite{bardenet2020monte} (good properties for integration)
    and
 \cite{tremblay2019determinantal}  (asymptotic rebalancing) also hold for any
  sufficiently smooth kernel in the flat limit, at least for L-ensembles of fixed-size $m \in \magicn{d}$.
\end{remark}

The case of kernels with finite smoothness is different but still simply written if $m$ is greater than 
$\PP_{r-1,d}$:
\begin{corollary}
  \label{cor:nd-finite-smooth-magic-case}
  Let $d\in \mathbb{N}^*$ and $\bL_\varepsilon =
  [\kappa_\varepsilon(\vect{x}_i,\vect{x}_j)]_{i,j}$ for $\kappa$ a stationary
  kernel of smoothness order $r$ and $\vect{x}_1, \ldots, \vect{x}_n$
  vectors in $\R^d$. Then, for all $m > \PP_{r-1,d}$, the limiting distribution
  of $\Xe \sim |DPP|_m(\bL(\varepsilon))$ is:
  \[ \Xs \sim  \mppDPP{m} \ELE{(-1)^r \bD^{(2r-1)}} {\bV_{< r}} \]
\end{corollary}
\begin{proof}
 Subcase of Theorem~\ref{thm:general-case-smooth-fixed-size}.
\end{proof}

With these two corollaries in hand, we can go back to the teaser (figure
\ref{fig:teaser}) we gave in the introduction. In figure
\ref{fig:teaser}, the points 1 to 6 are on a parabolic curve: $x_2 = x_1^2$,
while point 7 ($x_1 = 0.5,x_2 = 0.6$) is not. For now let $X = \{1,2,3,4,5,6 \}$
and $X' = \{2,3,4,5,6,7 \}$. 
Applying corollary
\ref{cor:nd-smooth-magic-case} for a  $\mDPP{6}$ with a Gaussian kernel, we see
that $p(\Xs = X) \propto \det V_{\leq 2}(X)^2 = 0$
(the matrix is square and has two identical columns). On the other hand, one may check
numerically that $\det V_{\leq 2}(X')$ is non-zero, even though $X'$ is less
spread-out than $X$. 
For the case of the exponential kernel, we apply corollary
\ref{cor:nd-finite-smooth-magic-case}, and we can verify numerically that
$X$ is much more likely than $X'$.
In fact, the two corollaries tell us more: the case of the Gaussian kernel holds in
fact for all kernels with $r>1$, which all give zero probability to set $X$. The
more general phenomenon this illustrates is that DPPs defined from smooth
kernels avoid non-unisolvent\footnote{Basically, $N\geq  \PP_{m,d}$ points $x_1,\ldots,x_N$ are unisolvent for polynomials of degree $m$ if the zero polynomial is the only one vanishing on all of them.}  sets, even though they may be acceptably
spread-out.

\subsection{The general case.}
\label{sec:non-universal-limits}

Up to here, we have covered all the easy cases which lead to universal limits. To be precise, for a fixed $d\in\mathbb{N}^*$ and $r\in\mathbb{N}^*$:
\begin{itemize}
	\item Corollary~\ref{cor:nd-finite-smooth-magic-case} covers the case $m> \PP_{r-1,d}$
	\item Out of the remaining cases where $m\leq \PP_{r-1,d}$, corollary~\ref{cor:nd-smooth-magic-case} covers the special cases where $m\in\magicn{d}$: $m=\PP_{0,d}$, $m=\PP_{1,d}$, $\ldots$, $m=\PP_{r-1,d}$.
\end{itemize}
 What remains is to cover the not-so-easy cases where $m\leq \PP_{r-1,d}$ and
 $m\notin\magicn{d}$. The statement of the results involves derivatives of the kernel.
A convenient short-hand notation for higher-order derivatives uses
multi-indices:
\[
  f^{(\va)}
  (\vect{x}) =\frac{\partial f^{|\va|}}{\partial x_1^{\alpha_1} \cdots\partial x_d^{\alpha_d}} (\vect{x})
\]
The Wronskian matrix of the kernel is defined as:
\begin{equation}\label{eq:wronskian_nd}
  \matr{W}_{\leq k}  = 
  \left[
    \frac{k^{(\va,\vect{\beta})} (\vect{0},\vect{0})}{\va!\vect{\beta}!}
  \right]_{|\va| \leq k, |\vb| \leq k} \in\mathbb{R}^{\PP_{k,d}\times \PP_{k,d}}.
\end{equation}
Here we index the matrix using multi-indices (equivalently, monomials), so that an element of
$\matr{W}_{\leq k}$ is e.g., $\matr{W}_{(0,2),(2,1)}$ which is a scaled
derivative of $k(\vect{x},\vect{y})$ of order $(0,2)$ in $\vect{x}$ and $(2,1)$
in $\vect{y}$.
For example, for $d=2$ and $k=2$ we may write 
\[
\matr{W}_{\leq 2} =\begin{bmatrix}
k^{((0,0),(0,0))} & k^{((0,0),(1,0))} & k^{((0,0),(0,1))} & \frac{k^{((0,0),(2,0))}}{2} & {k^{((0,0),(1,1))}} & \frac{k^{((0,0),(0,2))}}{2} \\
k^{((1,0),(0,0))} & k^{((1,0),(1,0))} & k^{((1,0),(0,1))} & \frac{k^{((1,0),(2,0))}}{2} & {k^{((1,0),(1,1))}} & \frac{k^{((1,0),(0,2))}}{2} \\
k^{((0,1),(0,0))} & k^{((0,1),(1,0))} & k^{((0,1),(0,1))} & \frac{k^{((0,1),(2,0))}}{2} & {k^{((0,1),(1,1))}} & \frac{k^{((0,1),(0,2))}}{2} \\
\frac{k^{((2,0),(0,0))}}{2} & \frac{k^{((2,0),(1,0))}}{2} & \frac{k^{((2,0),(0,1))}}{2} & \frac{k^{((2,0),(2,0))}}{4} & {\frac{k^{((2,0),(1,1))}}{2}} & \frac{k^{((2,0),(0,2))}}{4} \\
k^{((1,1),(0,0))} & k^{((1,1),(1,0))} & k^{((1,1),(0,1))} & \frac{k^{((1,1),(2,0))}}{2} & {k^{((1,1),(1,1))}} & \frac{k^{((1,1),(0,2))}}{2} \\
\frac{k^{((0,2),(0,0))}}{2} & \frac{k^{((0,2),(1,0))}}{2} & \frac{k^{((0,2),(0,1))}}{2} & \frac{k^{((0,2),(2,0))}}{4} & {\frac{k^{((0,2),(1,1))}}{2}} & \frac{k^{((0,2),(0,2))}}{4} \\
\end{bmatrix}\in\mathbb{R}^{\PP_{2,2}\times \PP_{2,2}}
\]
for a given ordering of the monomials, and where all the derivatives are taken
at $\vect{x}=0,\vect{y}=0$.

\begin{corollary}
  \label{cor:general-case-smooth-fixed-size}
  Let $d\in \mathbb{N}^*$ and $\bL_\varepsilon =
  [\kappa_\varepsilon(\vect{x}_i,\vect{x}_j)]_{i,j}$ for $\kappa$ a stationary
  kernel of smoothness order $r$, and $\vect{x}_1, \ldots, \vect{x}_n$
  vectors in $\R^d$. Let $m< \PP_{r-1,d}$ and $k\leq r-1$ the integer such
  that $\PP_{k-1,d} < m <  \PP_{k,d}$. Let us partition the Wronskian
  $\bW_{< k}$ as: 
  \[
  \matr{W}_{< k} =  \begin{bmatrix}\matr{W}_{< k-1} & \matr{W}_{\left\urcorner\right.} \\ \matr{W}_{\llcorner}  &\matr{W}_{\lrcorner}  \end{bmatrix} .
  \]
  Then, the limiting distribution of $\Xe \sim |DPP|_m(\bL_\varepsilon)$ is:
  \[\Xs \sim \mppDPP{m} \ELE{\bV_k\bar{\bW}\bV_k^\top}{\bV_{< k}}\]
  where $\bar{\bW}\in\mathbb{R}^{\HH_{k,d}\times \HH_{k,d}}$ is the  Schur complement:
  \[ \bar{\bW}= \matr{W}_{\lrcorner} - \matr{W}_{\llcorner}
    (\matr{W}_{< k - 1})^{-1}\matr{W}_{\urcorner} \]
\end{corollary}
\begin{proof}
	Subcase of Theorem~\ref{thm:general-case-smooth-fixed-size}.
\end{proof}

\subsection{Numerical illustrations}
\label{sec:numerics-nd}

We show here some numerical results analoguous to those of section~
\ref{sec:numerics-1d}.  In figures \ref{fig:cond-dens-2d-exp} and
\ref{fig:cond-dens-2d-r2}, we show the convergence of conditional densities for
two different kernels.  We illustrate the conditional probabilities of $\vect{x} \cup Y  \large| Y $ where $Y$ comprises seven points already sampled. Even if the ground set is finite and for the sake of illustration,  $\vect{x}$ varies continuously in the unit square. Figure \ref{fig:convergence-inclusion-2d} shows the convergence of inclusion probabilities in an example.

\begin{figure}[h]
	\centering
	\includegraphics[width=12cm]{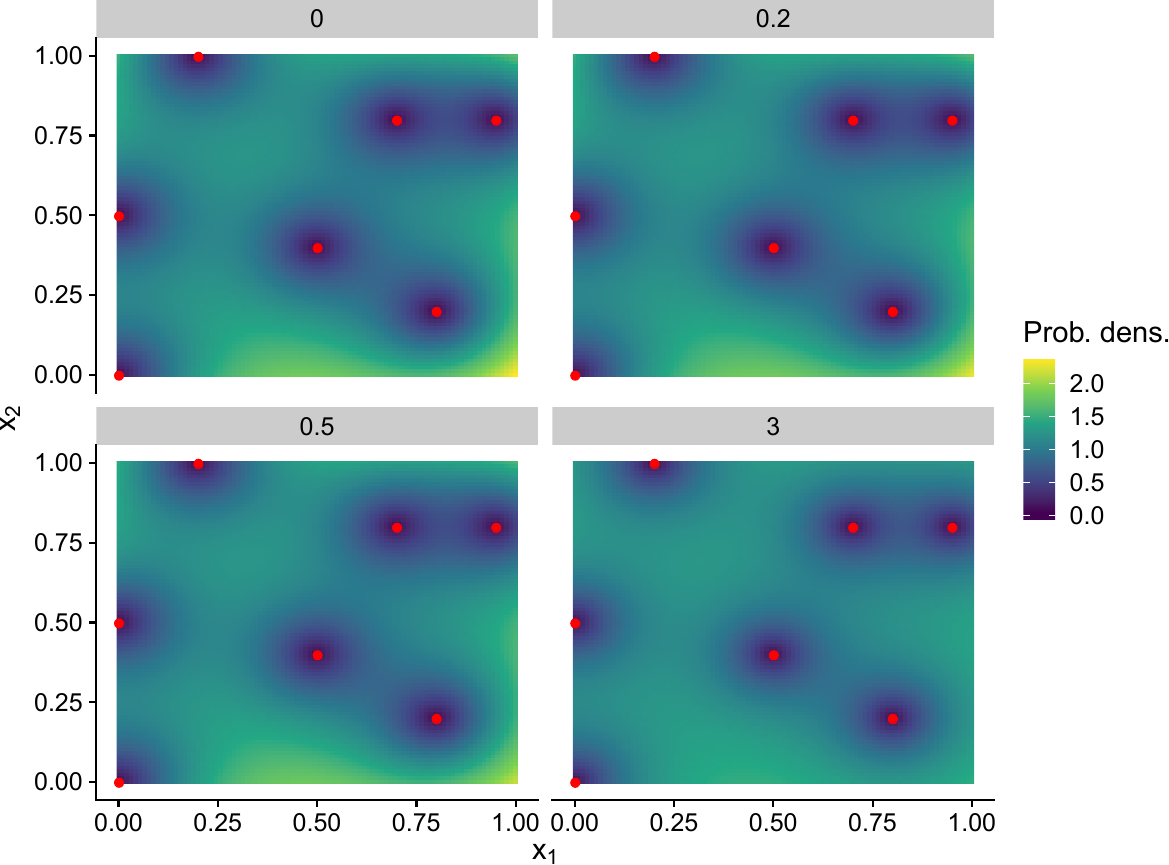}
	\caption{Conditional probability density for $\vect{x}\in [0,1]^2$ conditional on the 7
		nodes in red, for the exponential kernel $\exp(-\norm{\vect{x} -
			\vect{y}})$. The four panels represent the density for different values of
		$\varepsilon$ (panels are labelled with the value). The top-left panel is the theoretical limit.
	}
	\label{fig:cond-dens-2d-exp}
\end{figure}

\begin{figure}[h]
	\centering

	\includegraphics[width=12cm]{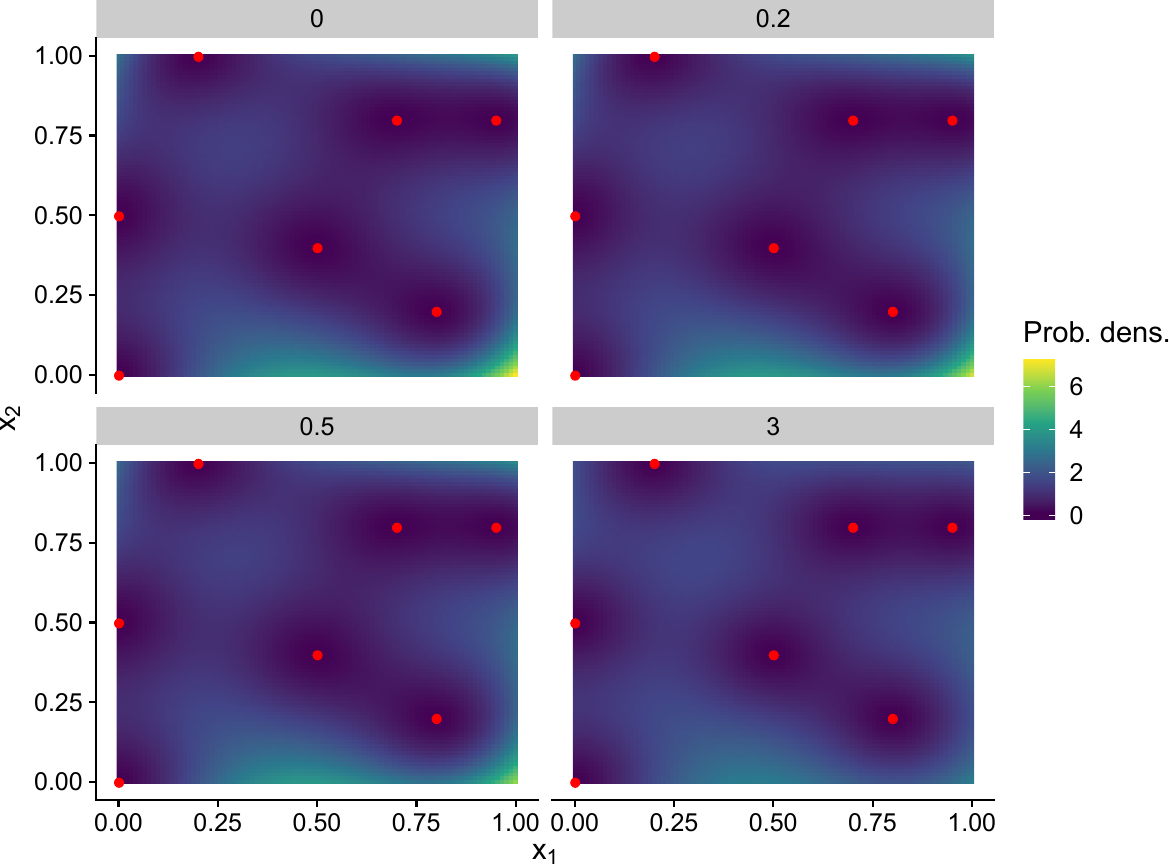}
	\caption{Same as in figure \ref{fig:cond-dens-2d-exp}, but for the kernel $(1+\norm{\vect{x} - \vect{y}})\exp(-\norm{\vect{x} - \vect{y}})$
	}
	\label{fig:cond-dens-2d-r2}
\end{figure}

\begin{figure}[h]
		\includegraphics[width=12cm]{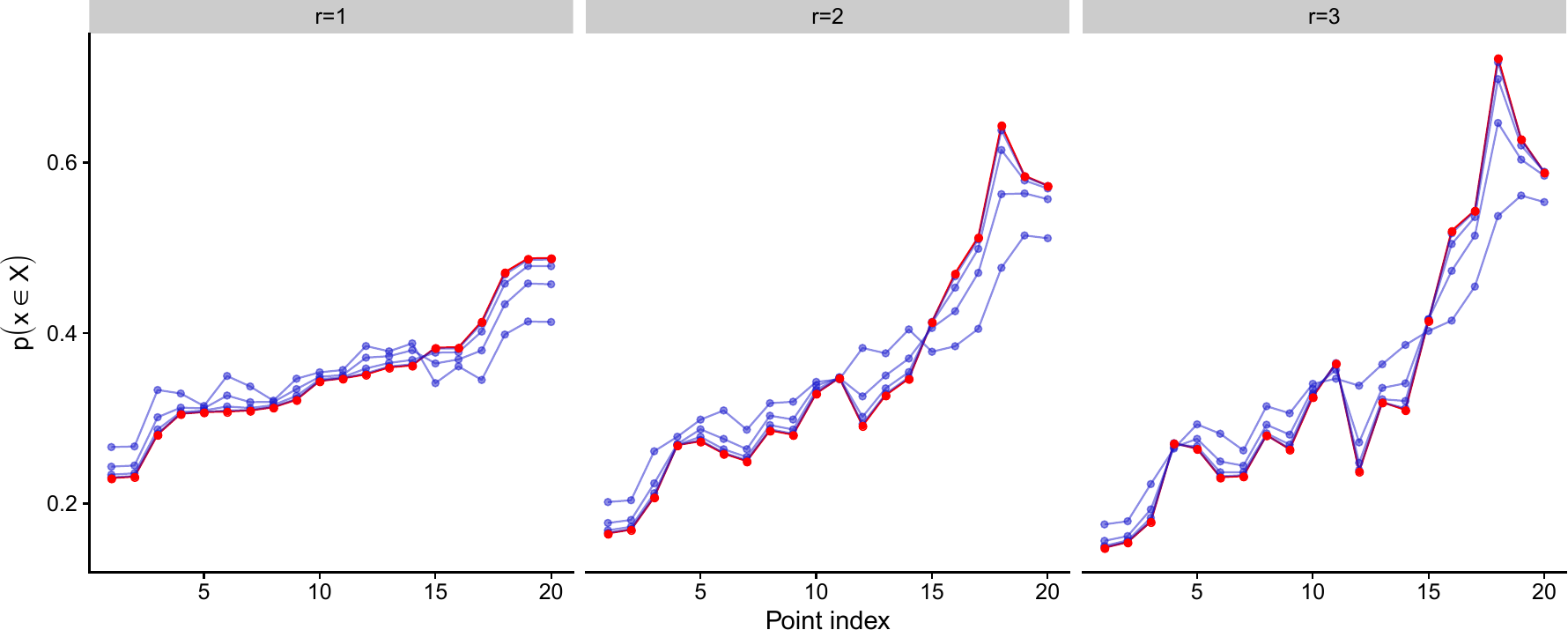}
	\caption{Flat limit of inclusion probabilities of (fixed-size) L-ensembles for three different
		kernels, multivariate case. Here we plot $\Proba(x \in \Xe)$, the inclusion probabilities for a fixed size L-ensemble 
		(with $m=7$), where the ground set $\Omega$ consists in 20 points drawn at
		random from the unit square. To better visualise the convergence, we plot
		$\Proba(x_i \in \Xe)$ as a function of the index $i$, and we have ordered the
		points according to their inclusion probability for the first kernel.
		Everything else is analoguous to fig. \ref{fig:convergence-inclusion-1d}. 
		The dots in blue (joined by lines for
		clarity) are inclusion probabilities for 
		$\varepsilon = 4,1.5,.5,.1$. The dots in red represent
		the limit in $\flatlim$. The three kernels are, from left-to-right,
		$\exp(-\delta), (1+\delta)\exp(-\delta),(3+3\delta+\delta^2)\exp(-\delta)$,
		where $\delta = \norm{\vect{x} - \vect{y}}$. These kernels have $r=1$,$2$ and $3$, respectively.}
	\label{fig:convergence-inclusion-2d}
\end{figure}

\section{The flat limit of fixed-size and varying-size L-ensembles: two general theorems}
\label{sec:two_general_theorems}
\subsection{The fixed-size case}
This is the general theorem that was thoroughly discussed in the form of several corollaries in the last two sections.
\begin{theorem}
	\label{thm:general-case-smooth-fixed-size}
	Let $d\in \mathbb{N}^*$ and $\Omega=\{\vect{x}_1,\ldots,\vect{x}_n\}$ a set of $n$
	distinct points in $\R^d$. Let $\bL_\varepsilon =
	[\kappa_\varepsilon(\vect{x}_i,\vect{x}_j)]_{i,j}$ for $\kappa$ a stationary
	kernel of smoothness order $r$. Let the integer $m\leq n$ be the number of desired samples. There are three possible scenarios depending on the value of $m$:
	\begin{enumerate}
		\item for all $m \leq \PP_{r-1,d}~$ verifying $m\in\magicn{d}$, \emph{i.e.}, for all values of $m$ for which there exists an integer $k\leq r-1$ such that $m=\PP_{k,d}$, the fixed-size L-ensemble $\X_\varepsilon \sim
		|DPP|_m(\bL_\varepsilon)$ has the limiting distribution, as $\flatlim$:
		\[ \Xs \sim |DPP|_m(\bV_{\leq k}\bV_{\leq k}^\top)\]
		Equivalently, if $\bQ\in\mathbb{R}^{n\times \PP_{k,d}}$ is an orthonormal basis for $\bV_{\leq k}$, then:
		\begin{align}
		\label{eq:with_Q}
		\Xs \sim |DPP|_m(\bQ\bQ^\top)
		\end{align}
		\item for all $m \leq \PP_{r-1,d}~$ verifying $m\notin\magicn{d}$, \emph{i.e.}, for all values of $m$ for which there does not exist $k\leq r-1$ such that $m=\PP_{k,d}$, the following is verified. Denote by $k\leq r-1$ the integer such
		that $\PP_{k-1,d} < m <  \PP_{k,d}$. Let us partition the Wronskian
		$\bW_{< k}$ as: 
		\[
		\matr{W}_{< k} =  \begin{bmatrix}\matr{W}_{< k-1} & \matr{W}_{\left\urcorner\right.} \\ \matr{W}_{\llcorner}  &\matr{W}_{\lrcorner}  \end{bmatrix} .
		\]
		Then, the limiting distribution of $\Xe \sim |DPP|_m(\bL_\varepsilon)$ is:
		\[\Xs \sim \mppDPP{m} \ELE{\bV_k\bar{\bW}\bV_k^\top}{\bV_{< k}}\]
		where $\bar{\bW}\in\mathbb{R}^{\HH_{k,d}\times \HH_{k,d}}$ is the  Schur complement:
		\[ \bar{\bW}= \matr{W}_{\lrcorner} - \matr{W}_{\llcorner}
		(\matr{W}_{< k - 1})^{-1}\matr{W}_{\urcorner} \]
		\item for all $m > \PP_{r-1,d}$, the limiting distribution
		of $\Xe \sim |DPP|_m(\bL_\varepsilon)$ is:
		\[ \Xs \sim  \mppDPP{m} \ELE{(-1)^r \bD^{(2r-1)}} {\bV_{< r}} \]
	\end{enumerate}
\end{theorem}
\begin{proof}
	In the following, $\bL_{\varepsilon,\X}$ stands for the matrix $\bL_{\varepsilon}$ reduced to its lines and columns indexed by $\X$. 
	Below is the proof for the univariate case. The proof for the multivariate
case is in the same spirit, only more complicated: it can be found in Appendix\ref{appA}.\\
	
	In dimension 1, note that $\forall k, \PP_{k-1,d=1}=k$ such that all integers smaller than $\PP_{r-1,1}$ are in $\magicn{1}$: the second scenario never happens and only two scenarios exist: 1.~$m\leq r$ or 2.~$m> r$. We prove the results of both scenarios separately:
	\begin{enumerate}
		\item ($m \leq \PP_{r-1,1}=r$). Let $X$ be a set of $m$ points. As $m\leq r$, one can show that the determinant of $\bL_{\varepsilon,X}$ has the  expansion
		\begin{equation}		
		\det (\bL_{\varepsilon,X}) = \varepsilon^{m(m-1)} (\det(\matr{V}_{< m}(X))^2 \det \matr{W}_{< m} + \O(\varepsilon)).\label{eq:det_1d_smooth}
		\end{equation}
		where we have made explicit in the notation the quantities that depend on the points
		$X$ versus those that do not. 
		This result originally appeared in \cite{lee2015flatkernel}, and can be found in this form in theorem 4.1 of~\cite{BarthelmeUsevich:KernelsFlatLimit}. 
		Now, let $\bL^\star = \bV_{< m}\bV_{< m}^\top$. The previous expansion implies:
		\[ \Proba(\X_\varepsilon = X) =  \frac{  \varepsilon^{m(m-1)} \left( \det
			\matr{W}_{< m}  \det
			\bL^\star_{X} + \O(\varepsilon)  \right) }{ \varepsilon^{m(m-1)} \left(
			\det \matr{W}_{< m} \sum_{Y,|Y|=m} \det
			\bL^\star_{Y} + \O(\varepsilon)  \right)}
		\]
		We  may apply lemma \ref{lem:TV-convergence} directly: $\X_\varepsilon$ tends to
		$\X_\star$, a fixed-size DPP with law: 
		\[ \Proba(\X_\star = X) =  \frac{\det
			\bL^\star_{X} }{\sum_{Y,|Y|=m} \det
			\bL^\star_{Y} } \]
		
		\item ($m > \PP_{r-1,1}=r$). In this case, theorem 4.4 of~\cite{BarthelmeUsevich:KernelsFlatLimit} for kernels with lower order of smoothness yields the following expansion for the determinant of $\bL_{\varepsilon,X}$:
		\begin{equation}
		\det (\bL_{\varepsilon,X}) =\varepsilon^{m(2r-1)-r^2} \left( \widetilde{l}(X) + \O(\varepsilon)\right),
		\label{eq:det_finite_smoothness}
		\end{equation}
		where the main term is given by
		\begin{equation}
		\label{eq:ltilde}
		\widetilde{l}(X) =
		(-1)^{r} 
		\det \matr{W}_{< r}  \det
		\begin{bmatrix}
		f_{2r-1}  \bD^{(2r-1)}(X) &  \matr{V}_{< r}(X) \\
		\matr{V}_{< r}(X) ^\top& 0 
		\end{bmatrix}\\
		\end{equation}
		Using lemma \ref{lem:TV-convergence} as for the first scenario, one obtains the result.
	\end{enumerate}
\end{proof}

\subsection{The varying-size case}
The varying-size case is more involved, mainly due to the extra scaling
parameters that are required in order to obtain non-trivial limits. 
To see why rescaling is needed, consider the one-dimensional case and $r=1$
(e.g., the exponential kernel in $d=1$). If we set $\Xe \sim
DPP(\bL_\varepsilon)$, the size of $\Xe$ will be $\leq 1$ with probability 1 in
the limit. This follows from the scaling behaviour in $\varepsilon$ of the
determinants. In an L-ensemble,
\[ \Proba(|\Xe| = k) = \frac{\sum_{|X| = k} \det \bL_{\varepsilon, X}}{\sum_{X} \det
    \bL_{\varepsilon,X}} = \frac{e_k(\bL_\varepsilon)}{\det(\bI+\bL_\varepsilon)} \]
The results in \cite{BarthelmeUsevich:KernelsFlatLimit} imply that if $r=d=1$,
$e_1(\bL_\varepsilon)$ scales as $\O(1)$, while $e_k(\bL_\varepsilon)$ scales as
$\O(\varepsilon^{k-1})$ for $k>1$. This means in turn that
the probability of sampling a set of size $>1$ is vanishingly small in
$\varepsilon$. To have a reasonable limit, we must rescale the L-ensemble as
$\flatlim$. For instance, in this case, it is enough to take $\Xe \sim
DPP(\alpha \varepsilon^{-1} \bL_\varepsilon)$ for some $\alpha > 0$. A quick
calculation shows that with this rescaling the size in the limit is $\geq 1$
with probability 1, with the expected size depending on $\alpha$ and the
spectrum of $\bL$. The theorem below shows that different scalings lead to
different sizes in the limits. Interestingly, some scalings lead to (universal)
projection DPPs and others to (non-universal) partial projection DPPs. 

Another way to think about the rescaling introduced in the theorem below, is
that it is equivalent to taking the limit in $\flatlim$ while holding $\E |\Xe|
= m$
constant. Depending on the value of $m$, different scaling orders are required. This can be proved formally by appealing to the Newton-Puiseux
theorem, but we skip the argument here because it requires additional background.

\begin{theorem}
	\label{thm:Xe_varying_multivariate}
	Let $d\in\mathbb{N}^*$, $p\in\mathbb{N}$, $\alpha>0$, and $\Omega=\{\vect{x}_1,\ldots,\vect{x}_n\}$ a set of $n$
	distinct points in $\R^d$. Let $\bL_\varepsilon = [\kappa_\varepsilon(\vect{x}_i,\vect{x}_j)]_{i,j}$ with $\kappa$ a
	stationary kernel of smoothness order $r\in\mathbb{N}^*$.  Let $\Xe \sim
	DPP(\alpha\varepsilon^{-p}\bL_\varepsilon)$.  In the limit $\flatlim$, the
	distribution of $\Xe$ depends on the interplay between $p, r$ and $n$. First of all, $p$ is either even or odd: only one out of the two following values $\left(\frac{p}{2},\frac{p+1}{2}\right)$ is an integer. We call that integer $l$. Now, if $\PP_{l-1,d}\geq n$  then, for any value of $r$, $\Xe$ has limit $\Omega$ with probability one. 
	Otherwise, there are three scenarii depending on the value of $r$: 
	\begin{enumerate}
		\item if $r<\frac{p+1}{2}$, then $\Xe$ has limit $\Xs=\Omega$ with probability one.
		\item if $r>\frac{p+1}{2}$, $\Xe$ has a limiting distribution that depends on the parity of $p$:
		\begin{enumerate}
			\item If $p$ is odd ($l=\frac{p+1}{2}$), then $\Xe$ has limit $\Xs \sim \mDPP{\PP_{l-1,d}}(\bV_{< l}\bV_{< l}^\top)$
			\item If $p$ is even ($l=\frac{p}{2}$) then $\Xe$ has limit $\Xs \sim \ppDPP \ELE{\alpha \bV_{l}\bar{\bW}\bV_{l}^\top}{\bV_{<
					l}}$ with $\bV_{l}\bar{\bW}\bV_{l}^\top$ as in theorem \ref{thm:general-case-smooth-fixed-size}.
		\end{enumerate}
		\item if $r=\frac{p+1}{2}$, then $\Xe$ has limit $\Xs \sim DPP \ELE{\alpha f_{2r-1}
			\bD^{(2r-1)}}{\bV_{<  r}}$.
	\end{enumerate}
\end{theorem}

\begin{remark}
  The case $r>\frac{p+1}{2}$, $p$ odd is a universal limit, and a fixed-size
  DPP. For instance, if $r>1$, and $p=3$, then we obtain the same limiting DPP
  regardless of the specific kernel, and it has a fixed size equal to the number
  of monomials of degree $\leq 2$ in dimension $d$.
  Another almost-universal limit is given by the last case in the theorem
  where all kernels lead to the same limit up to rescaling by $\alpha$. If the
  kernel has $r=3$, for instance, we may take $p=2r-1=5$ to obtain this limit. 
\end{remark}

\begin{proof}
	See Appendix~\ref{appB}.
\end{proof}

\section{Practical consequences}
\label{sec:practical-consequences}

In \cite{paper1}, we suggested using a certain family of DPPs as a good
``default'' for sampling with repulsion. The family is defined by an extended
L-ensemble $\ELE{\bL}{\bV}$ with kernel
\begin{equation}
  \label{eq:default-kernel}
\bL =  \left[ \gamma (-1)^{\lceil \beta/2 \rceil} \norm{\vect{x}-\vect{y}}^\beta \right]_{\vect{x} \in
  \Omega,\vect{y} \in \Omega}
\end{equation}
and $\bV = \bV_{\leq \lceil \beta/2
  \rceil - 1}$ a multivariate Vandermonde matrix for monomials of degree $\leq
\lceil \beta/2 \rceil - 1$. There are only two hyperparameters: $ \beta
\not\in 2\mathbb{N}$, which controls the amount of repulsion, and $\gamma$,
which controls the expected size. This particular family of DPPs can be
motivated by appealing to the theory of conditional positive definite kernels,
but another justification follows from Theorem
\ref{thm:Xe_varying_multivariate}: it is also the process obtained by taking the
flat limit of finitely-smooth L-ensembles, for instance the Matèrn family of
kernels. Different choices of $\beta$ correspond to different values of the
smoothness order $r$ \footnote{Formally, this is valid for $\beta$ integer
  and odd, and when using the scaling $\varepsilon^{2r-1}\bL$}. 

We expect this aspect of our results to be useful to practitioners. When using
the Matèrn family of kernels to define L-ensembles, one would have to deal with
three hyperparameters; one for smoothness order ($r$), one for spatial
length-scale $\varepsilon$, and one for expected size ($\gamma$). How repulsive
the process is depends on both $\varepsilon$ and $r$ in a complicated and
unpredictable manner. Eliminating $\varepsilon$ by taking the flat limit yields
a repulsive point process with only two hyperparameters. 

When a DPP is defined instead using the Gaussian kernel, there are only two
hyperparameters, one for spatial scale and one for expected size. The spatial
scale controls the repulsiveness. By theorem \ref{thm:Xe_varying_multivariate},
eliminating $\varepsilon$ through the flat limit yields a DPP that is still
repulsive, but based on global basis functions (the discrete orthogonal
polynomials of $\Omega$). There is a single hyperparameter left, the size.
Because of the particular structure induced by the set of basis functions, there
are natural choices for setting $\E | \X| = p$ in dimension $d$. If one picks
one of the ``magic'' numbers for $p$, then the flat limit is a projection DPP.
In dimension 2 for example, there are 6 monomials of degree $\leq 2$, and 10 of
degree $\leq 3$. If we pick $p=10$ then a projection DPP may be defined using
the basis functions $\bV_{\leq 3}$, of which there are ten. Taking $p=11$
implies taking into account just part of the fourth-degree monomials. There is
no canonical way to do this, and this is reflected in Theorem
\ref{thm:Xe_varying_multivariate}, which shows the flat limit is non-universal
in this case, and involves the Wronskian of the kernel. The magic numbers are
therefore ``natural'' choices for the sample size of a DPP in dimension $d$. 
They correspond to the (projection) DPPs $\mDPP{k+d \choose d}(\bV_{\leq
  k}\bV_{\leq k}^\top)$ for different values of $k$, called Vandermonde DPPs in
\cite{tremblay2019determinantal} and orthogonal polynomial ensembles in
\cite{bardenet2020monte}. 

To sum up, eliminating spatial scale removes a hyperparameter but results in a
sensible limit. If one starts from a family where $r$ is free, one is left with
two hyperparameters; if one starts from a family where $r=\infty$, only one
hyperparameter is left. Finally, notice that the DPPs $\mDPP{k+d \choose
  d}(\bV_{\leq   k}\bV_{\leq k}^\top)$ are a special case of the extended
L-ensembles defined by eq. \eqref{eq:default-kernel}, obtained by taking
$\gamma=0$, so that $\E |\X|$ is then determined by $\beta$. This suggests that
global interaction is determined by the polynomial basis functions, while the
kernel in eq. \eqref{eq:default-kernel} governs mostly local interactions (as
remark \ref{rem:exp-kernel-d1} hints when $d=1$). We believe that this may point
to computational savings, but we leave this for future work.

\section{To conclude}
\label{sec:discussion}

In the flat limit, L-ensembles formed from stationary kernels stay well-defined (and
meaningfully repulsive). In some cases we obtain universal limits where the
limit process depends only on $r$ and not the Wronskian of the kernel. In
dimension $d > 1$, these universal limits are obtained for certain natural
values of $m$ (for fixed-size L-ensembles) or when rescaling with $\varepsilon^{-p}$
for $p$ odd (varying-size L-ensembles).

The question of how fast L-ensembles converge to the limits given here requires
expansions to the next order, which we do not yet have. Empirically, we observe
that convergence is quite fast in the fixed-size case, but slower in the
varying-size case, at least in some instances. This means that the distribution
of the size of $\Xe$ may converge slowly to its limit. We hope to investigate
this further in future work.

In the interests of space we have left some topics aside. Our results on the
flat limit should apply as well to D-optimal design, and there is an interesting
connection to polyharmonic splines for kernels
with finite $r$ (see \cite{song2012multivariate,fanuel2020determinantal}). We have also entirely skipped
the topic of computational applications of these results. Finally, the
univariate results point to possible connections with random matrix theory we
have yet to explore.

Directions for future work include extending the results to continuous DPPs, and
in a related vein letting $n \rightarrow \infty$ as $\flatlim$ in discrete DPPs.
This should let one take advantage of some results from the literature on the
asymptotics of Christoffel functions, as in \cite{tremblay2019determinantal}. It
would also be worth investigating the flat limit on Riemannian manifolds, rather
than on $\R^d$ as we do here.

\begin{acks}[Acknowledgments]
	We thank Guillaume Gautier for helpful comments on preliminary versions of this
	manuscript.
\end{acks}

\begin{funding}
	This work was supported by the ANR projects GenGP (ANR-16-CE23-0008), GRANOLA (ANR-21-CE48-0009), and 
	LeaFleT (ANR-19-CE23-0021-01), as well as the LabEx PERSYVAL-Lab (ANR-11-LABX-0025-01),
	the Grenoble Data Institute (ANR-15-IDEX- 02), MIAI@Grenoble Alpes
	(ANR-19-P3IA-0003), the LIA CNRS/Melbourne Univ Geodesic, and the IRS (Initiatives de Recherche Stratégiques) of the IDEX Université Grenoble Alpes. 
\end{funding}
  
\begin{appendix}
 
  \section{Proof of Theorem 5.1 in the Multivariate Case}
  \label{appA}
  We prove all three scenarios separately.
  \begin{proof}
  	\begin{enumerate}
  		\item ($m \leq \PP_{r-1,d}$ and $m\in \magicn{d}$) Case 1 of theorem 6.1 in \cite{BarthelmeUsevich:KernelsFlatLimit} states the behavior in $\varepsilon$ of the determinant in this case:
  		\begin{align*}
  		\forall X \text{ s.t. } | X|=m,\qquad \det(\bL_{\varepsilon,X}) = \varepsilon^{M} \left(\det \bW_{\leq k}(\det\bV_{\leq k}(X))^2 + \O(\varepsilon)\right) 
  		\end{align*} 
  		for some $M\in\mathbb{N}$ that we do not need to specify in this proof. The Wronskian matrix $\bW_{\leq
  			k}$ is irrelevant here as it does
  		not depend on $X$. Similarly to the univariate proof, one obtains that the limiting distribution is indeed $\Xs \sim |DPP|_m(\bV_{\leq k}\bV_{\leq k}^\top)$.
  		The equivalence with the formulation of Eq.~\eqref{eq:with_Q} comes
  		from applying Lemma~\ref{lem:max-rank-dpp} as $\bV_{\leq k}\bV_{\leq k}^\top$ is of rank $m$.
  		\item ($m \leq \PP_{r-1,d}$ and $m\notin \magicn{d}$) This is the most involved case. Let $X \subset\Omega$ be a subset of size $m$. 
  		Case 2 of theorem 6.1 in \cite{BarthelmeUsevich:KernelsFlatLimit} states the behavior in $\varepsilon$ of the determinant in this case:
  		\begin{equation}
  		\label{eq:limiting-det-general-case}
  		\det{\matr{L}_{\varepsilon,X}} = \varepsilon^{2 s(k,d)}  (\det(\matr{Y}\matr{W}_{\leq k} \matr{Y}^{\T}) \det(\matr{V}_{< k}(X)^{\T} \matr{V}_{< k}(X)) + \O(\varepsilon))
  		\end{equation}
  		with $s(k,d) = d {k+d \choose d+1} - k (\PP_{k,d} - m)$ and $\matr{Y} \in \RR^{m \times \PP_{k,d}}$ defined as:
  		\begin{align}
  		\label{eq:def_Y}
  		\matr{Y} = 
  		\begin{bmatrix}
  		\matr{I}_{\PP_{k-1,d}} &  \\ & \Qort(X)^{\T} \matr{V}_k(X)
  		\end{bmatrix},
  		\end{align}
  		$\matr{I}_{\PP_{k-1,d}}$ being the identity matrix of dimension $\PP_{k-1,d}$, $\Qort(X)\in\mathbb{R}^{m\times (m-\PP_{k-1,d})}$ is an orthonormal basis for the space orthogonal to $\text{span } \matr{V}_{< k}(X)$. 
  		
  		Expanding the expression, one obtains
  		\[
  		\det(\matr{Y}\matr{W}_{\leq k} \matr{Y}^{\T}) =
  		\det
  		\left(\begin{array}{c|c}
  		\matr{W}_{< k} & \matr{W}_{\urcorner} \bV_k(X)^\top \Qort(X) \\ \hline
  		\Qort(X)^\top \bV_k(X)  \matr{W}_{\llcorner} & \Qort(X)^\top \bV_k(X)  \matr{W}_{\lrcorner}
  		\bV_k(X)^\top \Qort(X)
  		\end{array}\right).
  		\]
  		Applying lemma \ref{lem:block-det}, $\det(\matr{Y}\matr{W}_{\leq k} \matr{Y}^{\T})$ may be written:
  		\begin{align*}
  		&\det(\matr{W}_{< k})\det\left( \Qort(X)^\top \bV_k(X) \left( \matr{W}_{\urcorner}
  		- \matr{W}_{\llcorner}\matr{W}_{< k}^{-1}\matr{W}_{\urcorner} \right)\bV_k(X)^\top \Qort(X) \right) \\
  		&= \det(\matr{W}_{< k})\det\left( \Qort(X)^\top \bV_k(X) \bar{\bW} \bV_k(X)^\top \Qort(X) \right).
  		\end{align*}
  		Injecting into \eqref{eq:limiting-det-general-case} and applying lemma
  		\ref{lem:det-saddlepoint}, we obtain:
  		\[ \det{\matr{L}_{\varepsilon,X}} = \varepsilon^{2s(k,d)}
  		\left( \det(\matr{W}_{< k})\det
  		\begin{pmatrix}
  		\bV_k(X) \bar{\bW} \bV_k(X)^\top & \bV_{< k}(X) \\
  		\bV_{< k}(X)^\top  & \matr{0}
  		\end{pmatrix} + \O(\varepsilon)
  		\right).
  		\]
  		Applying lemma \ref{lem:TV-convergence} finishes the proof.
  		\item ($m > \PP_{r-1,d}$) Case 1 of theorem 6.3 in \cite{BarthelmeUsevich:KernelsFlatLimit} states the behavior in $\varepsilon$ of the determinant in this case:
  		\begin{align*}
  		\forall X \text{ s.t. } | X|=m\geq\PP_{r-1,d},\qquad \det (\matr{L}_{\varepsilon,  X}) =\varepsilon^{M} \left( \widetilde{l}( X) + \O(\varepsilon)\right),
  		\end{align*} 
  		with $\widetilde{l}(X)$ as in Eq.~\ref{eq:ltilde} (with
  		$\bD^{(2r-1)}(X)$, $\bW_{< r}$ and  $\matr{V}_{< r}(X)$ replaced
  		by their multivariate equivalent), and $M$ an integer. Similarly to the univariate proof, one obtains that the limiting distribution is indeed $\Xs \sim  \mppDPP{m} \ELE{(-1)^r\bD^{(2r-1)}} {\bV_{< r}}$.
  	\end{enumerate}
  \end{proof}
  
  \section{Proof of Theorem 5.2 }
  \label{appB}
  
  The proof of this theorem is quite involved and starts by showing a technical
  lemma on the limiting distribution of the \textit{size} of $\Xe$. It generalises
  the example we gave in the introduction to this section, to determine the
  expected size  $\E |\X|$ as a function of $r$, $n$, $d$ and the parameters of
  the scaling. 
  \subsection{A technical lemma}
  \begin{lemma}
  	\label{lem:size_Xe_varying_multivariate}
  	Let $d\in\mathbb{N}^*$, $p\in\mathbb{N}$, $\alpha>0$, and $\Omega=\{\vect{x}_1,\ldots,\vect{x}_n\}$ a set of $n$
  	distinct points in $\R^d$. Let $\bL_\varepsilon = [\kappa_\varepsilon(\vect{x}_i,\vect{x}_j)]_{i,j}$ with $\kappa$ a
  	stationary kernel of smoothness order $r\in\mathbb{N}^*$.  Let $\Xe \sim
  	DPP(\alpha\varepsilon^{-p}\bL_\varepsilon)$.  In the limit $\flatlim$, the distribution of the size of $\Xe$ depends on the interplay between $p, r$ and $n$. First of all, $p$ is either even or odd: only one out of the two following values $\left(\frac{p}{2},\frac{p+1}{2}\right)$ is an integer. We call that integer $l$. Now, if $\PP_{l-1,d}\geq n$  then, for any value of $r$, as $\flatlim$, $|\Xe|=n$ with probability one. 
  	Otherwise, there are three scenarii depending on the value of $r$: 
  	\begin{enumerate}
  		\item if $r<\frac{p+1}{2}$, then, as $\flatlim$, $|\Xe|=n$ with probability one.
  		\item if $r>\frac{p+1}{2}$, the size of $\Xe$ has a distribution that depends on the parity of $p$:
  		\begin{enumerate}
  			\item If $p$ is odd ($l=\frac{p+1}{2}$), then, as $\flatlim$, $|\Xe|=\PP_{l-1,d}$ with probability one. 
  			\item If $p$ is even ($l=\frac{p}{2}$) then, as $\flatlim$, the distribution tends to:
  			\begin{align*}
  			\Proba(|\Xs| = m) =  \left\{ \begin{array}{ll} 
  			0 & \text{ if } m < \PP_{l-1,d} \text{ or } m > \PP_{l,d}\\ 
  			\frac{e_{m -
  					\PP_{l-1,d}}(\alpha \widetilde{\bV_{l}\bar{\bW}\bV_{l}^\top})}{\det\left(\matr{I}+\alpha\widetilde{\bV_l \bar{\bW} \bV_l^\top}\right)} & otherwise\end{array} \right. 
  			\end{align*}
  			where $\bV_{l}\bar{\bW}\bV_{l}^\top$ is as in theorem \ref{thm:general-case-smooth-fixed-size}, and $\widetilde{\bV_{l}\bar{\bW}\bV_{l}^\top}=(\bI-\bQ_l\bQ_l^\top) \bV_{l}\bar{\bW}\bV_{l}^\top(\bI-\bQ_l\bQ_l^\top)$, $\bQ_l$ being an orthonormal basis of ${\rm{span}}(\matr{V}_{< l})$.
  		\end{enumerate}
  		\item if $r=\frac{p+1}{2}$, then, as $\flatlim$, the distribution tends to:
  		\begin{align}
  		\label{distrib:size_Xe_r=_multivar}
  		\Proba(|\Xs| = m) =  \left\{ \begin{array}{ll} 
  		0 & \mbox{if }  m< \PP_{r-1,d} \mbox{ or } m>n\\ 
  		\frac{e_{m-\PP_{r-1,d}}\left(\alpha f_{2r-1}\widetilde{\bD^{(2r-1)}}\right)}{\det\left(\bI+\alpha f_{2r-1}\widetilde{\bD^{(2r-1)}}\right)} & \mbox{otherwise}
  		\end{array} \right. 
  		\end{align}
  		where $\widetilde{\bD^{(2r-1)}}=(\bI-\bQ_r\bQ_r^\top) \bD^{(2r-1)}(\bI-\bQ_r\bQ_r^\top)$, $\bQ_r$ being an orthonormal basis of ${\rm{span}}(\matr{V}_{< r})$.
  	\end{enumerate}
  \end{lemma}
  
  \begin{proof}
  	Recall that $\bL_{\varepsilon,\X}$ stands for the matrix $\bL_{\varepsilon}$ reduced to its lines and columns indexed by $\X$. 
  	First, recall that if $\X \sim DPP(\bL)$, then the marginal distribution of its size $|\X|$ is given by (see, \textit{e.g.}, Corollary 1.18 of~\cite{paper1})
  	\begin{equation}
  	\label{eq:marginal-distr-size_multivar}
  	\Proba(|\X| = m) =  \frac{e_m(\bL)}{e_0(\bL) + e_1(\bL) + \ldots + e_n(\bL)}.
  	\end{equation}
  	where $e_m(\bL)$ is the $m$-th elementary symmetric polynomial of $\bL$ and for
  	consistency $e_0(\bL)=1$ for all matrices $\bL$. Recall also that $e_i(\bL)$  verifies:
  	\begin{align*}
  	e_i(\bL)=\sum_{|\X|=i} \det \bL_{\X}.
  	\end{align*}
  	Here, we consider the $L$-ensemble $\alpha\varepsilon^{-p}\bL_\varepsilon$. 
  	One has $\det(\alpha\varepsilon^{-p} \bL_{\varepsilon,\X}) =
  	\alpha^{|\X|}\varepsilon^{-p|\X|}\det(\bL_{\varepsilon,\X})$, which yields $\forall i$: $ e_i(\alpha\varepsilon^{-p} \bL_\varepsilon) = \alpha^i \varepsilon^{-ip} e_i(\bL_\varepsilon)$. \\	
  	Let $r\in\mathbb{N}^*$, $d\in\mathbb{N}^*$ and consider $i\leq \PP_{r-1,d}$. In the flat limit, we can apply theorem 6.1 in~\cite{BarthelmeUsevich:KernelsFlatLimit}. There are two cases: either $i$ is a magic number ($i\in\magicn{d}$) in which case $k\in\mathbb{N}$ will denote the integer verifying $i=\PP_{k,d}$, or it is a muggle number ($i\notin\magicn{d}$) in which case $k\in\mathbb{N}$ denotes the smallest integer such that $i<\PP_{k,d}$. In both cases, we denote by $M(i)$ the integer $M(i)=d {k+d \choose d+1}$. Combining points 1 and 2 of Theorem 6.1 in~\cite{BarthelmeUsevich:KernelsFlatLimit}, one has, $\forall \; 1\leq i\leq \PP_{r-1,d}$:
  	\begin{align*}
  	e_i(\bL_\varepsilon) &= \sum_{|\X|=i} \det \bL_{\varepsilon,\X} \\
  	&= \varepsilon^{2\left(M(i)+k(\PP_{k,d}-i)\right)} \left( \sum_{|\X|=i} \det(\matr{Y}\matr{W}_{\leq k}\matr{Y}^\top) \det(\matr{V}_{< k}(\X)^\top \matr{V}_{< k}(\X)) + \O(\varepsilon) \right)
  	\end{align*}
  	where $\matr{Y}$ is as in Eq.~\eqref{eq:def_Y}. Let us define $\tilde{e}_i$ by 
  	\begin{align}
  	\label{eq:etilde_multivar}
  	\forall\; 1\leq i\leq \PP_{r-1,d}\qquad \tilde{e}_i&=\sum_{|\X|=i} \det(\matr{Y}\matr{W}_{\leq k}\matr{Y}^\top) \det(\matr{V}_{< k}(\X)^\top \matr{V}_{< k}(\X))
  	\end{align}
  	such that:
  	\begin{align*}
  	\forall\; 1\leq i\leq \PP_{r-1,d}\qquad e_i(\bL_\varepsilon)=  \varepsilon^{2\left(M(i)+k(\PP_{k,d}-i)\right)} \left( \tilde{e}_i + \O(\varepsilon)  \right).
  	\end{align*}
  	Also, we can apply theorem 6.3 of~\cite{BarthelmeUsevich:KernelsFlatLimit} for any set $\X$ of size $i> \PP_{r-1,d}$:
  	\begin{align*}
  	\forall i> \PP_{r-1,d}\qquad e_i(\bL_\varepsilon) &= \sum_{|\X|=i} \det \bL_{\varepsilon,\X} = \varepsilon^{2d  {r+d-1 \choose d+1} +(2r-1)(i-\PP_{r-1,d})} \left( \sum_{|\X|=i} \tilde{l}(\X) + \O(\varepsilon) \right)\\
  	&=\varepsilon^{2d  {r+d-1 \choose d+1} +(2r-1)(i-\PP_{r-1,d})} \left( \bar{e}_i + \O(\varepsilon) \right)
  	\end{align*}
  	where $\bar{e}_i$ verifies:
  	\begin{align}
  	\label{eq:ebar}
  	\bar{e}_i=\sum_{|X|=i} \tilde{l}(X)=(-1)^r\det(\matr{W}_{<r})\sum_{|X|=i}\det\begin{bmatrix}
  	f_{2r-1}  \bD^{(2r-1)}(X) &  \matr{V}_{< r}(X) \\
  	\matr{V}_{< r}(X)^{\top} & 0 
  	\end{bmatrix}
  	\end{align}
  	Now, injecting into Eq.~\eqref{eq:marginal-distr-size_multivar}, one shows that $\Proba(|\X| = m)$ may be written as:
  	\begin{align*}
  	\forall m,\qquad \Proba(|\X| = m) &= \frac{\varepsilon^{\eta(m)}\left(f_0(m) + \O(\varepsilon)  \right) }{\sum_{i=0}^n \varepsilon^{\eta(i)}\left(f_0(i) + \O(\varepsilon)  \right) }
  	\end{align*}
  	where $\eta(\cdot)$ and $f_0(\cdot)$ are two $\varepsilon$-independent functions verifying:
  	\begin{equation}
  	\label{eq:d_of_i_vsize_proof_multivar}
  	\eta(i) = \left\{ \begin{array}{ll} 
  	\eta_0(i)=0 & \mbox{if }  i=0 \\ 
  	\eta_1(i)=i(2-p)-2 & \mbox{if }  0<i\leq \PP_{1,d} \\ 
  	\eta_2(i)=i(4-p)-2d-4 & \mbox{if }  \PP_{1,d}\leq i\leq \PP_{2,d} \\ 
  	\vdots\\
  	\eta_l(i)=i(2l-p)-2 {d+l \choose d+1} & \mbox{if }  \PP_{l-1,d}\leq i\leq \PP_{l,d} \\ 
  	\vdots\\
  	\eta_{r-1}(i)=i(2r-2-p)-2 {d+r-1 \choose d+1} & \mbox{if }  \PP_{r-2,d}\leq i\leq \PP_{r-1,d} \\ 
  	\eta_{r}(i)=i(2r-1-p)- \left(2+\frac{d+1}{r-1}\right){d+r-1 \choose d+1} & \mbox{if }  i\geq \PP_{r-1,d} 
  	\end{array} \right. 
  	\end{equation}
  	and
  	\begin{align*}
  	f_0(i) = \left\{ \begin{array}{ll} 
  	1 & \mbox{if }  i=0 \\ 
  	\alpha^i \tilde{e}_i & \mbox{if }  0<i\leq \PP_{r-1,d} \\ 
  	\alpha^i \bar{e}_i & \mbox{if } i> \PP_{r-1,d}.
  	\end{array} \right. 
  	\end{align*}
  	We now make use of lemma~\ref{lem:TV-convergence-diverging}. In order to apply it, one needs to find the integers between $0$ and $n$ for which $\eta(\cdot)$ is minimal:
  	$$\argmin_{i\in\mathbb{N}, i\in[0,n]} \eta(i).$$ 
  	The answer to this question depends on $p,
  	r$ and $n$ which explains the different cases of the theorem. 	
  	Let us make first
  	a few simple observations on the function $\eta:\mathbb{R}^+\rightarrow \mathbb{R}$ :
  	\begin{itemize}
  		\item $\eta(\cdot)$ is continuous (everywhere except in $i=0$) and piecewise linear.
  		\item the slope of each of the linear pieces of $\eta(\cdot)$ is strictly increasing, starting at $2-p$ for the first piece $0<i\leq \PP_{1,d}$ and finishing at $2r-1-p$ for the last piece $i\geq \PP_{r-1,d}$.
  	\end{itemize}
  	We shall now explore all the possible cases sequentially. 
  	\begin{enumerate}
  		\item if $r<\frac{p+1}{2}$, \textit{i.e.}, $2r-1-p<0$: the slope of all the pieces of $\eta(\cdot)$ are negative, and $\eta(\cdot)$ is thus strictly decreasing. In this case, the integer in $[0,n]$ minimizing $\eta$ is $i=n$. Applying lemma~\ref{lem:TV-convergence-diverging}, as $\flatlim$, $|\Xe|=n$ with probability $1$.
  		\item if $r>\frac{p+1}{2}$:
  		\begin{enumerate}
  			\item if $p$ is odd, then $\frac{p-1}{2}$ is an integer and $\PP_{\frac{p-1}{2},d}$ is well defined. Trivially, $r>\frac{p+1}{2}$ implies $\PP_{\frac{p-1}{2},d}<\PP_{r-1,d}$. Also, note that $\eta(\cdot)$ decreases strictly between $0^+$ and $\PP_{\frac{p-1}{2},d}$, and then increases strictly after $\PP_{\frac{p-1}{2},d}$. The integer in the interval $[0,n]$ minimizing $\eta(\cdot)$ is thus $\min\left(\PP_{\frac{p-1}{2},d},n\right)$. Applying lemma~\ref{lem:TV-convergence-diverging}, as $\flatlim$, $|\Xe|=\min\left(\PP_{\frac{p-1}{2},d},n\right)$ with probability $1$.
  			\item if $p$ is even (the case $p=0$ falls into this category, recall that $\PP_{-1,d}$ is by convention set to $0$), then $r > \frac{p+1}{2}$ implies $\frac{p}{2}\leq r-1$ and thus  $\PP_{\frac{p}{2},d}\leq\PP_{r-1,d}$. Also, note that $\eta(\cdot)$ decreases strictly between $0^+$ and $\PP_{\frac{p}{2}-1,d}$, is constant between $\PP_{\frac{p}{2}-1,d}$ and $\PP_{\frac{p}{2},d}$, and then increases strictly after $\PP_{\frac{p}{2},d}$. The integers in the interval $[0,n]$ minimizing $\eta(\cdot)$ are thus: 
  			\begin{itemize}
  				\item $\{n\}$ if $n\leq  \PP_{\frac{p}{2}-1,d}$. In this case, applying lemma~\ref{lem:TV-convergence-diverging}, as $\flatlim$, $|\Xe|=n$ with probability $1$.
  				\item all the integers contained in the interval $\left[\PP_{\frac{p}{2}-1,d},\min\left(\PP_{\frac{p}{2},d},n\right)\right]$ if $n\geq  \PP_{\frac{p}{2}-1,d}$. In the following $I_{p,d}$ is the list of these integers. 
  				Applying lemma~\ref{lem:TV-convergence-diverging}, as $\flatlim$:
  				\begin{align}
  				\forall m\in I_{p,d}\quad	\Proba(|\Xs|=m)=\frac{\alpha^m\tilde{e}_m}{\sum_{i\in I_{p,d}} \alpha^i\tilde{e}_i}
  				\end{align}
  				Now, using the same arguments as in the proof of theorem~\ref{thm:general-case-smooth-fixed-size}, note that $\tilde{e}_i$ defined in  Eq.~\eqref{eq:etilde_multivar} may be re-written as:
  				\begin{align}
  				\label{eq:sum_over_|X|}
  				\forall i\in I_{p,d}\qquad \tilde{e}_i=
  				\det(\matr{W}_{< k})\sum_{|\X|=i} \det
  				\begin{pmatrix}
  				\bV_k(\X) \bar{\bW} \bV_k(\X)^\top & \bV_{< k}(\X) \\
  				\bV_{< k}(\X)^\top  & \matr{0}
  				\end{pmatrix}
  				\end{align}
  				where $\bar{\bW}$ is as in theorem~\ref{thm:general-case-smooth-fixed-size}. 
  				Now, consider the NNP $(\bV_k \bar{\bW} \bV_k^\top; \matr{V}_{< k})$ as well as $\widetilde{\bV_k \bar{\bW} \bV_k^\top}$ as defined in Definition~\ref{def:ext_Lens}. Note that the rank of $\widetilde{\bV_k \bar{\bW} \bV_k^\top}$ is $\min\left(\HH_{k,d}, n-\PP_{k-1,d}\right)$. One recognizes in the sum over $|\X|=i$ in Eq.~\eqref{eq:sum_over_|X|} the normalization constant of the fixed-size extended L-ensemble associated to this NNP and obtains, for all integer $i\in I_{p,d}$:
  				\begin{align}
  				\tilde{e}_i=
  				\det(\matr{W}_{< k}) (-1)^{\PP_{k-1,d}} e_{i-{\PP_{k-1,d}}}\left(\widetilde{\bV_k \bar{\bW} \bV_k^\top}\right) \det \left(\left(\matr{V}_{< k}\right)^\top \matr{V}_{< k}\right).
  				\end{align}
  				Simplifying, one obtains:
  				\begin{align}
  				\label{eq:temp_result}
  				\forall m\in I_{p,d}\quad \Proba(|\Xs|=m)=\frac{e_{m-{\PP_{k-1,d}}}\left(\alpha\widetilde{\bV_k \bar{\bW} \bV_k^\top}\right)}{\sum_{i\in I_{p,d}}e_{i-{\PP_{k-1,d}}}\left(\alpha\widetilde{\bV_k \bar{\bW} \bV_k^\top}\right)}.
  				\end{align}
  				Changing the summing index gives:
  				\begin{align}
  				\sum_{i\in I_{p,d}}e_{i-{\PP_{k-1,d}}}\left(\alpha\widetilde{\bV_k \bar{\bW} \bV_k^\top}\right)=
  				\sum_{i=0}^{\min(\HH_{k,d},n-\PP_{k-1,d})} e_{i}\left(\alpha\widetilde{\bV_k \bar{\bW} \bV_k^\top}\right).
  				\end{align}
  				Finally, note that, as $\text{rank}\left(\widetilde{\bV_k \bar{\bW} \bV_k^\top}\right) = \min\left(\HH_{k,d}, n-\PP_{k-1,d}\right)$, all the elementary symmetric polynomials $e_i$ for $i>\min\left(\HH_{k,d}, n-\PP_{k-1,d}\right)$ are null. The denominator of Eq.~\eqref{eq:temp_result} is thus $\sum_{i=0}^n e_{i}\left(\alpha\widetilde{\bV_k \bar{\bW} \bV_k^\top}\right)=\det\left(\matr{I}+\alpha\widetilde{\bV_k \bar{\bW} \bV_k^\top}\right)$ and one obtains:
  				\begin{align}
  				\forall m\in I_{p,d}\quad \Proba(|\Xs|=m)=\frac{e_{m-{\PP_{k-1,d}}}\left(\alpha\widetilde{\bV_k \bar{\bW} \bV_k^\top}\right)}{\det\left(\matr{I}+\alpha\widetilde{\bV_k \bar{\bW} \bV_k^\top}\right)}.
  				\end{align}
  			\end{itemize}
  		\end{enumerate}
  		\item if $r=\frac{p+1}{2}$, $\eta(\cdot)$ decreases strictly between $0^+$ and $\PP_{r-1,d}$, and is constant after $\PP_{r-1,d}$. The integers in the interval $[0,n]$ minimizing $\eta(\cdot)$ are thus: 
  		\begin{itemize}
  			\item $\{n\}$ if $n\leq  \PP_{r-1,d}$. In this case, applying lemma~\ref{lem:TV-convergence-diverging}, as $\flatlim$, $|\Xe|=n$ with probability $1$.
  			\item all those contained in the interval $\left[\PP_{r-1,d},n\right]$ if $n\geq  \PP_{r-1,d}$. In the following $I_{r,d}$ is the list of these integers. 
  			Applying lemma~\ref{lem:TV-convergence-diverging}, as $\flatlim$:
  			\begin{align}
  			\label{eq:I_rd_P}
  			\forall m\in I_{r,d}\quad	\Proba(|\Xs|=m)=\frac{\alpha^m\bar{e}_m}{\sum_{i\in I_{r,d}} \alpha^i\bar{e}_i}.
  			\end{align}
  			Now, consider the NNP $(f_{2r-1}\bD^{(2r-1)}; \matr{V}_{< r})$ as well as $f_{2r-1}\widetilde{\bD^{(2r-1)}}$ as defined in Definition~\ref{def:ext_Lens}. Note that the rank of $f_{2r-1}\widetilde{\bD^{(2r-1)}}$ is $n-\PP_{r-1,d}$. When looking at the definition of $\bar{e}_i$ in Eq.~\eqref{eq:ebar}, one recognizes in the sum over $|\X|=i$  the normalization constant of the fixed-size extended L-ensemble associated to this NNP and obtains, for all integer $i\in I_{r,d}$:
  			\begin{align*}
  			\sum_{|X|=i}\det\begin{bmatrix}
  			f_{2r-1}  \bD^{(2r-1)}(X) &  \matr{V}_{< r}(X) \\
  			\matr{V}_{< r}^{\T}(X) & 0 
  			\end{bmatrix} = (-1)^r e_{i-r}\left(f_{2r-1}\widetilde{\bD^{(2r-1)}}\right) \det \left(\left(\matr{V}_{< r}\right)^\top \matr{V}_{< r}\right)
  			\end{align*}
  			Injecting this in Eq.~\eqref{eq:ebar} and the resulting $\bar{e}_i$ in Eq.\eqref{eq:I_rd_P}, one obtains after simplifying in a similar fashion as in case 2:
  			\begin{align}
  			\Proba(|\Xs| = m) =  \left\{ \begin{array}{ll} 
  			0 & \mbox{if }  m< \PP_{r-1,d}\\ 
  			\frac{e_{m-\PP_{r-1,d}}\left(\alpha f_{2r-1}\widetilde{\bD^{(2r-1)}}\right)}{\det\left(\bI+\alpha f_{2r-1}\widetilde{\bD^{(2r-1)}}\right)} & \mbox{if } m\geq \PP_{r-1,d}\end{array} \right. 
  			\end{align}
  		\end{itemize}
  	\end{enumerate}
  	Finally, one may see that the three cases just described can in fact be equivalently stated in the form of the lemma, finishing the proof.
  \end{proof}
  \subsection{Proof of the theorem}
  \begin{proof}[Proof of theorem~\ref{thm:Xe_varying_multivariate}] We first prove the theorem for the univariate case. 
  	
  	\noindent\textbf{\emph{The univariate case ($d=1$)}}. 	
  	We will prove each case sequentially. First of all, for all the cases in Lemma ~\ref{lem:size_Xe_varying_multivariate} for which $|\Xe|=n$ in the limit $\flatlim$, the set $\Xe$ obviously tends to $\Omega$. Let us now focus on scenario number 2. 
  	
  	In the case 2a, we know from Lemma ~\ref{lem:size_Xe_varying_multivariate} that $|\Xs|=\frac{p+1}{2}$ with probability one. The limiting process is thus a fixed-size L-ensemble of size $l=\frac{p+1}{2}$. The fixed-size limit applies and theorem~\ref{thm:general-case-smooth-fixed-size} implies the result.
  	
  	Case 2b needs a bit more work.  First of all, define the integer $l=\frac{p}{2}$ and consider an orthonormal basis $\bQ\in\mathbb{R}^{n\times l}$ of $\text{span}(\bV_{\leq l-1})$. Also, consider the vector $\vect{q}_{l}$ such that $\bQ'=\left[\bQ | \vect{q}_{l}\right]\in\mathbb{R}^{n\times (l+1)}$ is an orthonormal basis for ${\rm span}(\bV_{\leq l})$. 
  	From Lemma~\ref{lem:size_Xe_varying_multivariate}, as $\PP_{l-1,d=1}=l$ and $\PP_{l,d=1}=l+1$, the limiting distribution of $|\Xs|$ has only two possible values: $l$ and $l+1$. Noting that in dimension one $\bar{\bW}$ is simply a scalar and $\bV_l$ a one-dimensional vector, a short calculation yields
  	\[ |\Xs| =
  	\begin{cases}
  	l \text{ with probability } \frac{1}{1+\alpha\gamma} \\
  	l + 1  \text{ with probability } \frac{\alpha\gamma}{1+\alpha\gamma}
  	\end{cases}
  	\]
  	with \begin{align}
  	\label{eq:gamma_orig}
  	\gamma = \frac{\det (\bV_{\leq l}^\top \bV_{\leq
  			l})\det \bW_{\leq l}}{\det (\bV_{\leq l-1}^\top \bV_{\leq
  			l-1})\det \bW_{\leq l-1}}
  	\end{align}
  	Now, using Theorem ~\ref{thm:general-case-smooth-fixed-size} (we are in case 1), one obtains that $\Xs$ is a mixture of two fixed-size L-ensembles: with probability
  	$\frac{1}{1+\alpha \gamma}$, it has size $l$ and distribution $\mDPP{l}(\bQ\bQ^\top)$, and with probability $\frac{\alpha\gamma}{1+\alpha\gamma}$,
  	it has size $l+1$ and distribution $\mDPP{l+1}(\bQ'\bQ'^\top)$. Looking at the mixture representation of pp-DPPs described in Corollary~3.3 of~\cite{paper1}
  	, one observes that this limiting distribution can be succinctly described as a pp-DPP $\Xs \sim DPP \ELE{\alpha\gamma\bQ'\bQ'^\top}{\bQ}$. 
  	Now, by the invariance property of remark~3.11 of~\cite{paper1}
  	, this is equivalent to $\Xs  \sim \ppDPP \ELE{\alpha\gamma\vect{q}_{l}\vect{q}_{l}^\top}{\bQ}$. Also, by the invariance
  	property of remark~3.10 of~\cite{paper1}
  	, this is in turn equivalent to $\Xs  \sim \ppDPP \ELE{\alpha\gamma\vect{q}_{l}\vect{q}_{l}^\top}{\bV_{< l}}$.
  	Finally, noting that 
  	\begin{align*}
  	\frac{\det (\bV_{\leq l}^\top \bV_{\leq l})}{\det (\bV_{\leq l-1}^\top \bV_{\leq l-1})}	\vect{q}_{l}\vect{q}_{l}^\top= 	\bV_{l}\bV_{l}^\top
  	\end{align*} 
  	and injecting in the expression of $\gamma$ of Eq.~\ref{eq:gamma_orig}, one obtains that
  	$
  	\gamma
  	\vect{q}_{l}\vect{q}_{l}^\top  
  	=\frac{\det \bW_{\leq l}}{\det \bW_{\leq l-1}}\bV_{l} \bV_{l}^\top
  	=\bV_{l} \bar{\bW} \bV_{l}^\top
  	$, 
  	finishing the proof that the limit in case 2b is $\Xs \sim \ppDPP \ELE{\alpha \bV_l\bar{\bW}\bV_l^\top}{\bV_{< l}}$.
  	
  	Let us finish with case 3. From a mixture point of view, the limiting process can be described by:
  	\begin{enumerate}
  		\item draw the size $m$ of the set according to case 3 of Lemma~\ref{lem:size_Xe_varying_multivariate}:
  		\begin{align}
  		\Proba(|\Xs| = m) =  \left\{ \begin{array}{ll} 
  		0 & \mbox{if }  m< r \text{ or } m>n\\ 
  		\frac{e_{m-r}\left(\alpha f_{2r-1}\widetilde{\bD^{(2r-1)}}\right)}{\det\left(\bI+\alpha f_{2r-1}\widetilde{\bD^{(2r-1)}}\right)} & \mbox{otherwise}
  		\end{array} \right.
  		\end{align}
  		where $\widetilde{\bD^{(2r-1)}}=(\bI-\bQ\bQ^\top) \bD^{(2r-1)}(\bI-\bQ\bQ^\top)$, $\bQ$ being an orthonormal basis of ${\rm{span}}(\matr{V}_{< r})$.
  		\item conditionally on the size, draw a fixed-size pp-DPP, which, according to theorem~\ref{thm:general-case-smooth-fixed-size} (we are in case 3), reads
  		$\Xs \sim \mppDPP{m} \ELE{(-1)^r\bD^{(2r-1)}} {\bV_{< r}}$. 
  	\end{enumerate}
  	Noting that  $\Xs \sim \mppDPP{m} \ELE{(-1)^r\bD^{(2r-1)}} {\bV_{< r}}$ is equivalent to $\Xs \sim \mppDPP{m} \ELE{\alpha f_{2r-1}\bD^{(2r-1)}} {\bV_{< r}}$, this mixture is precisely the mixture representation (see Corollary~3.3 of~\cite{paper1}) of $\Xs \sim \ppDPP \ELE{\alpha f_{2r-1}\bD^{(2r-1)}} {\bV_{< r}}$, ending the proof.\\
  	
  	\noindent\textbf{\emph{The multivariate case ($d>1$)}}. The multivariate proof is omitted, as it behaves exactly like the univariate one: in all three cases, use Lemma~\ref{lem:size_Xe_varying_multivariate} and theorem~\ref{thm:general-case-smooth-fixed-size} to describe the limiting process as a mixture, before observing that these mixtures are precisely the mixture representations of pp-DPPs.
  	
  	
  	
  	
  \end{proof}
\end{appendix}


\begin{thebibliography}{16}

\bibitem{bardenet2020monte}
\begin{barticle}[author]
\bauthor{\bsnm{Bardenet},~\bfnm{R{\'e}mi}\binits{R.}} \AND
  \bauthor{\bsnm{Hardy},~\bfnm{Adrien}\binits{A.}}
(\byear{2020}).
\btitle{Monte Carlo with determinantal point processes}.
\bjournal{The Annals of Applied Probability}
\bvolume{30}
\bpages{368--417}.
\end{barticle}
\endbibitem

\bibitem{Barthelme:AsEqFixedSizeDPP}
\begin{barticle}[author]
\bauthor{\bsnm{Barthelm{\'e}},~\bfnm{Simon}\binits{S.}},
  \bauthor{\bsnm{Amblard},~\bfnm{Pierre-Olivier}\binits{P.-O.}} \AND
  \bauthor{\bsnm{Tremblay},~\bfnm{Nicolas}\binits{N.}}
(\byear{2019}).
\btitle{Asymptotic Equivalence of Fixed-size and Varying-size Determinantal
  Point Processes}.
\bjournal{Bernoulli}.
\end{barticle}
\endbibitem

\bibitem{BarthelmeUsevich:KernelsFlatLimit}
\begin{barticle}[author]
\bauthor{\bsnm{Barthelm{\'e}},~\bfnm{Simon}\binits{S.}} \AND
  \bauthor{\bsnm{Usevich},~\bfnm{Konstantin}\binits{K.}}
(\byear{2021}).
\btitle{Spectral properties of kernel matrices in the flat limit}.
\bjournal{SIAM Journal on Matrix Analysis and Applications (arXiv:1910.14067)}
\bvolume{42}
\bpages{17--57}.
\end{barticle}
\endbibitem

\bibitem{suppl}
  \begin{barticle}[author]
    \bauthor{\bsnm{Barthelm{\'e}},~\bfnm{Simon}\binits{S.}},
    \bauthor{\bsnm{Tremblay},~\bfnm{Nicolas}\binits{N.}},
    \bauthor{\bsnm{Usevich},~\bfnm{Konstantin}\binits{K.}} \AND
    \bauthor{\bsnm{Amblard},~\bfnm{Pierre-Olivier}\binits{P.-O.}}
    (\byear{2021}).
    \btitle{Supplement to Determinantal Point Processes in the Flat Limit}.
    \bjournal{Bernoulli}.
  \end{barticle}
  \endbibitem


\bibitem{driscoll2002interpolation}
\begin{barticle}[author]
\bauthor{\bsnm{Driscoll},~\bfnm{Tobin~A}\binits{T.~A.}} \AND
  \bauthor{\bsnm{Fornberg},~\bfnm{Bengt}\binits{B.}}
(\byear{2002}).
\btitle{Interpolation in the limit of increasingly flat radial basis
  functions}.
\bjournal{Computers \& Mathematics with Applications}
\bvolume{43}
\bpages{413--422}.
\end{barticle}
\endbibitem

\bibitem{fanuel2020determinantal}
\begin{barticle}[author]
\bauthor{\bsnm{Fanuel},~\bfnm{Micha{\"e}l}\binits{M.}},
  \bauthor{\bsnm{Schreurs},~\bfnm{Joachim}\binits{J.}} \AND
  \bauthor{\bsnm{Suykens},~\bfnm{Johan~AK}\binits{J.~A.}}
(\byear{2020}).
\btitle{Determinantal Point Processes Implicitly Regularize Semi-parametric
  Regression Problems}.
\bjournal{arXiv preprint arXiv:2011.06964}.
\end{barticle}
\endbibitem

\bibitem{gasca2000polynomial}
\begin{barticle}[author]
\bauthor{\bsnm{Gasca},~\bfnm{Mariano}\binits{M.}} \AND
  \bauthor{\bsnm{Sauer},~\bfnm{Thomas}\binits{T.}}
(\byear{2000}).
\btitle{Polynomial interpolation in several variables}.
\bjournal{Advances in Computational Mathematics}
\bvolume{12}
\bpages{377}.
\end{barticle}
\endbibitem

\bibitem{Gau20}
\begin{bphdthesis}[author]
\bauthor{\bsnm{Gautier},~\bfnm{Guillaume}\binits{G.}}
(\byear{2020}).
\btitle{{On sampling determinantal point processes}},
\btype{Ph.D. thesis},
\bpublisher{Ecole Centrale de Lille}.
\end{bphdthesis}
\endbibitem

\bibitem{KuleszaTaskar:FixedSizeDPPs}
\begin{binproceedings}[author]
\bauthor{\bsnm{Kulesza},~\bfnm{Alex}\binits{A.}} \AND
  \bauthor{\bsnm{Taskar},~\bfnm{Ben}\binits{B.}}
(\byear{2011}).
\btitle{k-DPPs: Fixed-size determinantal point processes}.
In \bbooktitle{Proceedings of the 28th International Conference on Machine
  Learning (ICML-11)}
\bpages{1193--1200}.
\end{binproceedings}
\endbibitem

\bibitem{kulesza2012determinantal}
\begin{barticle}[author]
\bauthor{\bsnm{Kulesza},~\bfnm{Alex}\binits{A.}},
  \bauthor{\bsnm{Taskar},~\bfnm{Ben}\binits{B.}} \betal{et~al.}
(\byear{2012}).
\btitle{Determinantal point processes for machine learning}.
\bjournal{Foundations and Trends{\textregistered} in Machine Learning}
\bvolume{5}
\bpages{123--286}.
\end{barticle}
\endbibitem

\bibitem{lee2015flatkernel}
\begin{barticle}[author]
\bauthor{\bsnm{Lee},~\bfnm{Yeon~Ju}\binits{Y.~J.}},
  \bauthor{\bsnm{Micchelli},~\bfnm{Charles~A.}\binits{C.~A.}} \AND
  \bauthor{\bsnm{Yoon},~\bfnm{Jungho}\binits{J.}}
(\byear{2015}).
\btitle{A study on multivariate interpolation by increasingly flat kernel
  functions}.
\bjournal{Journal of Mathematical Analysis and Applications}
\bvolume{427}
\bpages{74--87}.
\end{barticle}
\endbibitem

\bibitem{Macchi:CoincidenceApproach}
\begin{barticle}[author]
\bauthor{\bsnm{Macchi},~\bfnm{Odile}\binits{O.}}
(\byear{1975}).
\btitle{The coincidence approach to stochastic point processes}.
\bjournal{Advances in Applied Probability}
\bvolume{7}
\bpages{83-122}.
\bdoi{10.2307/1425855}
\end{barticle}
\endbibitem

\bibitem{song2012multivariate}
\begin{barticle}[author]
\bauthor{\bsnm{Song},~\bfnm{Guohui}\binits{G.}},
  \bauthor{\bsnm{Riddle},~\bfnm{John}\binits{J.}},
  \bauthor{\bsnm{Fasshauer},~\bfnm{Gregory~E}\binits{G.~E.}} \AND
  \bauthor{\bsnm{Hickernell},~\bfnm{Fred~J}\binits{F.~J.}}
(\byear{2012}).
\btitle{Multivariate interpolation with increasingly flat radial basis
  functions of finite smoothness}.
\bjournal{Advances in Computational Mathematics}
\bvolume{36}
\bpages{485--501}.
\end{barticle}
\endbibitem

\bibitem{stein1999interpolation}
\begin{bbook}[author]
\bauthor{\bsnm{Stein},~\bfnm{Michael~L}\binits{M.~L.}}
(\byear{1999}).
\btitle{Interpolation of Spatial Data: Some Theory for Kriging}.
\bpublisher{Springer}.
\end{bbook}
\endbibitem

\bibitem{tremblay2019determinantal}
\begin{barticle}[author]
\bauthor{\bsnm{Tremblay},~\bfnm{Nicolas}\binits{N.}},
  \bauthor{\bsnm{Barthelm{\'e}},~\bfnm{Simon}\binits{S.}} \AND
  \bauthor{\bsnm{Amblard},~\bfnm{Pierre-Olivier}\binits{P.-O.}}
(\byear{2019}).
\btitle{Determinantal Point Processes for Coresets}.
\bjournal{Journal of Machine Learning Research}
\bvolume{20}
\bpages{1--70}.
\end{barticle}
\endbibitem

\bibitem{paper1}
\begin{barticle}[author]
\bauthor{\bsnm{Tremblay},~\bfnm{Nicolas}\binits{N.}},
  \bauthor{\bsnm{Barthelm{\'e}},~\bfnm{Simon}\binits{S.}},
  \bauthor{\bsnm{Usevich},~\bfnm{Konstantin}\binits{K.}} \AND
  \bauthor{\bsnm{Amblard},~\bfnm{Pierre-Olivier}\binits{P.-O.}}
(\byear{2022}).
\btitle{Extended L-ensembles: a new representation for Determinantal Point
  Processes}.
\bjournal{Accepted to Annals of Applied Probability}.
\end{barticle}
\endbibitem

\bibitem{wendland2004scattered}
\begin{bbook}[author]
\bauthor{\bsnm{Wendland},~\bfnm{Holger}\binits{H.}}
(\byear{2004}).
\btitle{Scattered data approximation}
\bvolume{17}.
\bpublisher{Cambridge university press}.
\end{bbook}
\endbibitem

\end{thebibliography}

\end{document}